\documentclass[psreqno,reqno,11pt]{amsart}
\usepackage{amssymb}
\usepackage{amsmath}
\usepackage{comment}
\usepackage[english]{babel}

\usepackage{hyperref}

\theoremstyle{plain}
\newtheorem{theorem}{Theorem}[section]
\newtheorem{proposition}[theorem]{Proposition}
\newtheorem{lemma}[theorem]{Lemma}
\newtheorem{corollary}[theorem]{Corollary}

\theoremstyle{definition}
\newtheorem{definition}[theorem]{Definition}

\newtheorem{remark}[theorem]{Remark}
\newtheorem{notation}[theorem]{Notation}

\def\N{\mathbb{N}}

\def\P{\mathbb{P}}

\makeatletter
\@namedef{subjclassname@2020}{\textup{2020} Mathematics Subject Classification}
\makeatother

\title[Differentiable, Holomorphic, and Analytic Functions]{
	Differentiable, Holomorphic, and Analytic Functions on Complex $\Phi$-Algebras}

\author{M. Roelands}
\address{Mathematical Institute, Leiden University, 2300 RA Leiden,
	The Netherlands}
\email{m.roelands@math.leidenuniv.nl}

\author{C. Schwanke}
\address{Department of Mathematics and Applied Mathematics, University of Pretoria, Private Bag X20, Hatfield 0028, South Africa}
\email{cmschwanke26@gmail.com}
\date{\today}
\subjclass[2020]{Primary: 46A40; Secondary: 40J05}
\keywords{differentiation, complex vector lattice, analytic function, holomorphic function}

\begin{document}
	\def\sect#1{\section*{\leftline{\large\bf #1}}}
	\def\th#1{\noindent{\bf #1}\bgroup\it}
	\def\endth{\egroup\par}
	
	\begin{abstract}
		Using the notion of order convergent nets, we develop an order-theoretic approach to differentiable functions on Archimedean complex $\Phi$-algebras. Most notably, we improve the Cauchy-Hadamard formulas for universally complete complex vector lattices given by both authors in a previous paper in order to prove that analytic functions are holomorphic in this abstract setting. 
	\end{abstract}
	
	\maketitle

	\section{Introduction}\label{S:intro}
	
	The theory of complex analysis on complex Banach spaces enjoys a long, well-known, and fruitful history, which is summarized in the text \cite{Mujica}. This theory generalizes the classical theory of functions of a complex variable by using norms as an abstraction of the ordinary modulus on the complex plane. In this paper we investigate complex analysis via an alternative generalization of the complex modulus: \emph{the modulus on Archimedean complex vector lattices}. This complex vector lattice modulus naturally leads to the notion of order convergence, which provides us with an order-theoretic perspective on complex analysis. In particular, we introduce an order-theoretic approach to complex differentiation on Archimedean complex $\Phi$-algebras. 
	
	Specifically, in Section 3 we utilize the notion of order convergent nets in complex vector lattices to develop suitable definitions for order differentiable functions in Definition~\ref{D:order derivative}. It is then illustrated that order differentiable functions satisfy the classical sum, product, chain, and, when $E$ is uniformly complete, quotient rule. A conception of holomorphic functions in the present setting is also introduced.
	
	We then introduce in Section 4 the notions of super order differentiable functions and $\sigma$-super order differentiable functions as a means to strengthen our results on order continuity, the chain rule, and  quotient rule from Section 3.
	
	In Section 6 we apply our theory from Section 3 to power series and show that if a function is analytic in our abstract setting, then it is holomorphic. In order to perform this task, we first improve our Cauchy-Hadarmard formulas from \cite{RoeSch} in Section~\ref{S:CHF}, which now include spectra of convergence of $\Omega$ that are not necessarily order bounded. The key to this improvement is our obtainment of new formulas for the finite and infinite parts of a positive element in the sup-completion, which are introduced in \cite{AzouNasri}. Our alternative approach perhaps yields simpler formulas than the ones found in \cite{AzouNasri}, which facilitate our proofs for the new and improved Cauchy-Hadamard formulas.
	
	This paper continues our development of complex analysis on complex vector lattices, which began with a study of series and power series on such spaces in \cite{RoeSch}. The reader will infer that our order-theoretical approach reveals new insights, such as how the notions of weak order units and projection bands play a compelling role in complex analysis.
	
	\section{Preliminaries}\label{S:prelims}
	
	The reader is referred to the standard texts \cite{AlipBurk, Zaanen1, Zaanen2, littleZaanen} for any unexplained terminology or basic results in vector lattice theory. As usual, $\mathbb{C}$ denotes the field of complex numbers, and the set of strictly positive integers is denoted by $\mathbb{N}$. We also write $\mathbb{N}_{0}:=\mathbb{N}\cup\{0\}$ throughout.
	
	An Archimedean (real) vector lattice $F$ is said to be \textit{square mean closed} (see \cite[page~482]{AzBouBus} or \cite[page 356]{deSchipper}) if $\sup\{ (\cos\theta)x+(\sin\theta)y:\theta\in[0,2\pi]\}$
	exists in $F$ for every $x,y\in F$, in which case we write
	\[
	x\boxplus y:=\sup\{ (\cos\theta)x+(\sin\theta)y:\theta\in[0,2\pi]\}.
	\]
	
	Given a square mean closed Archimedean (real) vector lattice $F$, the vector space complexification $F\oplus iF$ is called an \textit{Archimedean complex vector lattice} \cite[pages 356--357]{deSchipper}, and $F$ is called the \textit{real part} of $F\oplus iF$.
	
	\begin{notation}
		We denote the real part of an Archimedean complex vector lattice $E$ by $F$ throughout this paper.
	\end{notation}
	
	The \textit{modulus} on an Archimedean complex vector lattice $E$ is defined by
	\[
	|x+iy|:=x\boxplus y\quad (x,y\in F).
	\]
	The \textit{positive cone} $E_+$ of $E$ is the set of all elements which are invariant under the modulus, that is, $E_{+}:=\{z\in E:|z|=z\}$. Observe that $E_+=F_+$.
	
	In \cite{RoeSch} it was necessary to confine our theory to the context of universally complete complex vector lattices. Indeed, without the assumption of universal completeness, fundamental results fail to be true. For example, it is shown in \cite[Remark~3.7]{RoeSch} that the geometric series can fail to converge in order in Archimedean complex vector lattices that are not universally complete. Hence we will focus solely on universally complete vector lattices in Sections~\ref{S:CHF} and \ref{S:power series}, where we study power series. We are able to relax the conditions on $E$ in Sections~\ref{S:diff'n} and \ref{S:super} to Archimedean complex $\Phi$-algebras, although some instances still require uniform completeness.
	
	If $F$ is an Archimedean \textit{(real) $\Phi$-algebra}, that is, an Archimedean (real) $f$-algebra possessing a multiplicative identity, then its multiplication canonically extends to a complex multiplication on $E=F\oplus iF$. The multiplication on $E$ will be indicated by juxtaposition, and we denote the multiplicative identity of $E$ by $e$ throughout. Therefore, $E$ is endowed with a complex $\Phi$-algebra structure and can rightly be called a \textit{complex $\Phi$-algebra}. Note that $E$ is semiprime by \cite[Corollary~10.4]{dP}.
	
	As usual, if $z\in E$ is (multiplicatively) invertible, we denote its multiplicative inverse by $z^{-1}$. Moreover, we set $z^0:=e$ for all $z\in E$. For more information on complex $\Phi$-algebras we refer the reader to \cite{BusSch} .
	
	We next record some useful basic properties of $E$, which we repeatedly and freely use throughout this paper. Note that real $\Phi$-algebra analogues of statements (i)--(iv) can be found in \cite[Section 142]{Zaanen2}.
	
	For all $z,w\in E$:
	
	\begin{itemize}
		\item[(i)] $zw=wz$.
		\item[(ii)] $|zw|=|z||w|$.
		\item[(iii)] $|z|\wedge|w|=0$ if and only if $zw=0$.
		\item[(iv)] If $z$ is invertible and positive then $z^{-1}$ is positive.
	\end{itemize}
	
	We next turn to the notion of order convergent nets in $E$.
	
	\begin{definition}\label{D:ord conv}
		A net $(z_\alpha)_\alpha$ in $E$ \textit{converges in order} to $z\in E$ (we denote this by writing $z_\alpha\to z$), if there exists a positive decreasing net $(q_\beta)_\beta$ with $\inf_{\beta}q_\beta=0$ (in symbols $q_\beta\downarrow 0$) 
		such that for all $\beta$, there exists $\alpha_{0}$ for which $|z_\alpha-z|\leq q_\beta$ for all $\alpha\geq\alpha_0$.
	\end{definition}
	
	We add that a natural analog of Definition~\ref{D:ord conv} for sequences is obtained by replacing the nets $(z_\alpha)_\alpha$ and $(q_\beta)_\beta$ with sequences $(z_n)_{n\ge 0}$ and $(q_m)_{m\ge 0}$. We will use the notation $z_n\to z$ to indicate when a sequence $(z_n)_{n\ge 0}$ in $E$ converges in order to $z\in E$.
	
	\section{Differentiation for order convergence}\label{S:diff'n}

	\begin{notation}
		$E$ will denote an Archimedean complex $\Phi$-algebra in this section.
	\end{notation}
	
	Next we extend the classical notion of differentiable functions on $\mathbb{C}$ to the Archimedean complex $\Phi$-algebra setting. The notation $f\colon\mathrm{dom}(f)\to E$ is used throughout, where by $\mathrm{dom}(f)$ we indicate the domain of $f$. 
	
	In the classical theory of differentiation, the points at which a function is differentiable are contained in some open set. For the order theoretical analogue, the situation is similar, but with subtle differences, as will be pointed out where necessary. To that end, we start by introducing the following notation.
	
	\begin{notation}
		For $z,w\in F$ we write $z\ll w$ to indicate when $w-z$ is a positive invertible element in $F$.
	\end{notation}
	
	The order theoretical analogues of open and closed disks are as follows.
	\begin{definition}
		For an invertible element $r \in E_+$ and $c\in E$ define 
		
		\[
		\overset{\circ}{\Delta}(c,r) := \{z \in E \colon |z-c| \ll r\},
		\]
		which plays the role of the analogue of an open disk. Moreover, we define 
		
		\[
		\bar{\Delta}(c,r):=\{z \in E \colon |z-c|\le r\},
		\]
		which plays the role of the analogue of a closed disk. We say for $c \in E$ and $r\in E_+$ an invertible element that a set of the form $U_c:=\overset{\circ}{\Delta}(c,r)$ is an \emph{order open neighborhood} of $c$. Furthermore, a set $U \subseteq E$ is said to be \emph{order open} if for every point $c \in U$ there is an order open neighborhood of $c$ contained in $U$.  
	\end{definition}
	
	\begin{remark}
		Note that for $c \in E$ and any invertible element $r\in E_+$ the order disk $\overset{\circ}{\Delta}(c,r)$ is order open, which is not the case for the set $\{z\in E\colon |z-c|<r\}$. Consider for example $\mathbb{C}^2$ with $r:=(1,1)$ and $c:=(0,0)$. Then $(1,0)<(1,1)$, but $\{(z,w)\in \mathbb{C}^2\colon |(z,w)|<(1,1)\}$ does not contain an order open neighborhood of $(1,0)$.   
	\end{remark}
	
	With the analogue of open disks in place, we introduce differentiable functions in the current setting.
	
	\begin{definition}\label{D:order derivative}
		A function $f\colon\mathrm{dom}(f)\to E$ is said to be \emph{order differentiable at $c\in\mathrm{dom}(f)$} if there exists an order open neighborhood $U_c\subseteq \mathrm{dom}(f)$ of $c$ and an $f_c\in E$ such that, for every $h_\alpha\to 0$ with $c+h_\alpha\in U_c$ for all $\alpha$, there is a net $q_\beta\downarrow 0$ with the property that for all $\beta$, there is an $\alpha_0$ for which
		\begin{align}\label{E:order derivative}
			\bigl|f(c+h_\alpha)-f(c)-h_\alpha f_c\bigr|\le|h_\alpha|q_\beta
		\end{align}
		holds whenever $\alpha\geq\alpha_0$. In this case we call $U_c$ a $D$-\textit{neighborhood} for $f$ at $c$.
	\end{definition}
	
	\begin{remark}
		In Definition~\ref{D:order derivative}, the order openness of $U_c$ implies that there exists an invertible element $r \in E_+$ such that $\overset{\circ}{\Delta}(c,r)\subseteq U_c$. Taking $(h_\alpha)_\alpha$ to be the sequence defined by $h_n:=\frac{1}{n+1}r$ consisting of positive invertible elements, we have $h_n\to 0$ and $c+h_n\in U_c$ for every $n\in\mathbb{N}_0$. Thus we can divide both sides of the inequality \eqref{E:order derivative} by $|h_n|=h_n$ to obtain
		\[
		\bigl|\bigl(f(c+h_n)-f(c)\bigr)h_n^{-1}-f_c\bigr|\le q_\beta,
		\]
		implying that $f_c$ is unique. 
	\end{remark}
	
	\begin{definition}\label{D:f'(c)}
		We call the unique $f_c$ in Definition~\ref{D:order derivative} the \emph{order derivative of $f$ at $c$} and will from now on denote it by $f'(c)$.
	\end{definition}
	
	\begin{remark}\label{R:not super diff}
		In Defintion~\ref{D:order derivative}, the assumption that \eqref{E:order derivative} holds for all $h_\alpha \to 0$ with $c+h_\alpha\in U_c$ does not imply in general that \eqref{E:order derivative} holds for all $h_\alpha \to 0$ with $c+h_\alpha\in\mathrm{dom}(f)$, as the following example illustrates.
		
		Consider the universally complete vector lattice of all complex-valued sequences $\mathbb{C}^\mathbb{N}$. Define $f\colon \mathbb{C}^\mathbb{N}\to\mathbb{C}^\mathbb{N}$ by
		\[
		f(z_1,z_2,z_3,...):=\begin{cases} (z_1,z_2,z_3,...)\qquad \textnormal{if}\qquad |(z_1,z_2,z_3,...)|\ll e\\
			(z_2,z_3,z_4,...)\qquad \textnormal{if}\qquad |(z_1,z_2,z_3,...)|\not\ll e.
		\end{cases}
		\]
		Let $h_\alpha\to 0$ satisfy $h_\alpha\in\overset{\circ}{\Delta}(0,e)$ for all $\alpha$. Then for each $\alpha$,
		\begin{align*}
			|f(0+h_\alpha)-f(0)-eh_\alpha|&=0\leq|h_\alpha|^2
		\end{align*}
		holds. Thus $f$ is order differentiable at $0$ and $f'(0)=e$.
		
		However, it is not true that
		\[
		|f(0+h_\alpha)-f(0)-h_\alpha|\leq|h_\alpha|q_\beta
		\]
		holds for some net $q_\beta\downarrow 0$ in $\mathbb{C}^\mathbb{N}_+$
		whenever $h_\alpha\in\mathrm{dom}(f)$ for all $\alpha$ and $h_\alpha\to 0$.
		
		To verify, suppose this statement is true. Note that the sequence
		\[
		h_n=(\underbrace{0,\dots,0}_{n\ \text{terms}},2,2,2,...)
		\]
		satisfies $h_n\to 0$ and $h_n\in\mathrm{dom}(f)$ for all $n\in\mathbb{N}_0$. Moreover, for $n\geq 2$ we have
		\begin{align*}
			|f(0+h_n)-f(0)-h_n|&=|f(h_n)-h_n|\\
			&=|(\underbrace{0,\dots,0}_{n-1\ \text{terms}},2,2,2,...)-(\underbrace{0,\dots,0}_{n\ \text{terms}},2,2,...)|\\
			&=|(\underbrace{0,\dots,0}_{n-1\ \text{terms}},2,0,0,0,...)|.
		\end{align*}
		
		Next fix $\beta$. Note that for any $n\geq 2$ we have 
		\[
		|(\underbrace{0,\dots,0}_{n-1\ \text{terms}},2,0,0,0,...)|\leq|(\underbrace{0,\dots,0}_{n\ \text{terms}},2,2,2,...)|q_\beta,
		\]
		which is a contradiction. Thus we will say that $f$ is not \emph{super order differentiable} at $0$, a topic which is explored in Section~\ref{S:super}.
	\end{remark}
	
	\begin{remark}\label{R:C^2}
		One can see from Remark~\ref{R:not super diff} that the function $f\colon\mathbb{C}^\mathbb{N}\to\mathbb{C}^\mathbb{N}$ defined by
		\[
		f(z_1,z_2,z_3,...):=(z_2,z_3,z_4,...)
		\]
		is not order differentiable at $0$. Similarly, one can show that the function $f\colon\mathbb{C}^2\to\mathbb{C}^2$ defined by $f(z,w):=(w,z)$ is not order differentiable at $0$ either. This fact displays a fundamental difference between order differentiablity and Fr\'echet differentiability, as $f$ has a Fr\'echet derivative everywhere on $\mathbb{C}^2$.
	\end{remark}
	
	The following lemma extends the classical result which states that differentiable complex functions are continuous. It will be strengthened to order continuity without restriction in Lemma~\ref{L:super diff is ord cont} using the notion of super order differentiability. 
	
	\begin{lemma}\label{L:diff is cont}
		If a function $f\colon \mathrm{dom}(f)\to E$ has order derivative at $c\in\mathrm{dom}(f)$, then there exists an order open neighborhood $U_c$ of $c$ such that the restriction $f|_{U_c}$ to $U_c$ is order continuous at $c$.
	\end{lemma}
	
	\begin{proof} Suppose $f$ has order derivative $f'(c)$ at $c$. Then there is an order open neighborhood $U_c$ of $c$ such that for $h_\alpha\to 0$ with $c+h_\alpha\in U_c$ for all $\alpha$, there exists a net $q_\beta\downarrow 0$ such that for every $\beta$ there exists an $\alpha_0$ for which 
		\[
		\bigl|f(c+h_\alpha)-f(c)-h_\alpha f'(c)\bigr|\le |h_\alpha|q_\beta\qquad (\alpha\geq\alpha_0).
		\]
		
		Let $(z_\alpha)_\alpha$ be a net in $U_c$ converging in order to $c$. Then the net $h_\alpha:=z_\alpha-c$ converges in order to $0$.
		
		Moreover, there is a net $p_\gamma\downarrow 0$ such that for all $\gamma$ there exists $\alpha_1$ such that $|h_\alpha|\le p_\gamma$ whenever $\alpha\geq\alpha_1$. For $\alpha_2\geq\alpha_0,\alpha_1$ we have 
		\begin{align*}
			|f(z_\alpha)-f(c)|&\le|f(c+(z_\alpha-c))-f(c)-(z_\alpha-c)f'(c)|+|z_\alpha-c||f'(c)|\\&\le |z_\alpha-c|q_\beta+|z_\alpha-c||f'(c)|\le p_\gamma(q_\beta+|f'(c)|).
		\end{align*}
		Hence $f(z_\alpha)\to f(c)$.
	\end{proof}
	
	We next shift our focus to differentiating the function $z \mapsto z^{-1}$, which is the foundation for the quotient rule. Denote by $E^{-1}$ the set of invertible elements of $E$. 
	
	\begin{remark}\label{R:conv iff bdd}  A net $(z_\alpha)_\alpha$ of invertible elements converging in order to an invertible element $z$ has the property that $z_\alpha^{-1}\to z^{-1}$ if and only if there exists an index $\alpha_0$ such that $\{|z_\alpha|^{-1}\colon \alpha\ge \alpha_0\}$ is order bounded. To verify this fact, note that if $\{|z_\alpha|^{-1}\colon \alpha\ge \alpha_0\}$ is order bounded for some $\alpha_0$, then for an upper bound  $u$ of $\{|z_\alpha|^{-1}\colon \alpha\ge \alpha_0\}$, we have
		\[
		\bigl|z_\alpha^{-1}-z^{-1}\bigr|=|z_\alpha|^{-1}|z|^{-1}|z_\alpha-z|\le u |z|^{-1}|z_\alpha-z|
		\]
		for all $\alpha\ge \alpha_0$, and so $z_\alpha^{-1}\to z^{-1}$.
		
		On the other hand, if $z_\alpha^{-1}\to z^{-1}$, then there is a net $p_\beta\downarrow 0$ such that for any $\beta$ there is an $\alpha_0$ such that $\bigl|z_\alpha^{-1}-z^{-1}\bigr|\le p_\beta$ for all $\alpha\ge\alpha_0$. Hence $|z_\alpha|^{-1}\le\bigl|z_\alpha^{-1}-z^{-1}\bigr|+|z|^{-1}\le p_\beta+|z|^{-1}$ for all $\alpha\ge \alpha_0$, so $\{|z_\alpha|^{-1}\colon \alpha\ge \alpha_0\}$ is order bounded.
	\end{remark}
	
	\begin{lemma}\label{L:locally ord cts} If $E$ is uniformly complete, then the set $E^{-1}$ is an order open set. In this case, considering the function $f\colon E^{-1}\to E^{-1}$ defined by $f(z):=z^{-1}$ and $c\in E^{-1}$, there exists an order open neighborhood $U_c\subseteq E^{-1}$ of $c$ such that $f|_{U_c}$ is order continuous.
	\end{lemma}
	
	\begin{proof}
		Assume that $E$ is uniformly complete. To show that $E^{-1}$ is order open, let $c\in E^{-1}$. If $z\in U_c:=\overset{\circ}{\Delta}(c,\frac{1}{2}|c|)$, then $\frac{1}{2}|c|\leq|z|$. Thus $|z|$ is invertible by \cite[Theorems~11.1~\&~11.4]{dP}. Then $z$ is also invertible with $z^{-1}=\bar{z}|z|^{-2}$, where the complex conjugate $\bar{z}$ is defined as usual. It follows that $E^{-1}$ is order open.
		
		Moreover, if $(z_\alpha)_\alpha$ is a net in $\overset{\circ}{\Delta}(c,\frac{1}{2}|c|)$ such that $z_\alpha \to c$, then $|z_\alpha|^{-1}\leq 2|c|^{-1}$ for every $\alpha$, so $z_\alpha^{-1}\to c^{-1}$ as explained in Remark~\ref{R:conv iff bdd}. Hence $f|_{U_c}$ is order continuous.
	\end{proof}
	
	\begin{remark}
		The set $E^{-1}$ is not necessarily order open if $E$ is not uniformly complete. For example, first consider the Archimedean real $\Phi$-algebra $F=PP[0,1]$ of all continuous piecewise polynomials on $[0,1]$. Let $F_1:=F$, and for $n\in\mathbb{N}$, let $F_{n+1}$ be the Archimedean real $\Phi$-algebra of all functions on $[0,1]$ that are of the form
		\[
		f(x):=\sum_{k=1}^m p_k(x)\sqrt{q_k(x)}
		\]
		for some $m\in\mathbb{N}$, $p_k\in F_n$, and $q_k\in (F_n)_+\ (k\in\{1,...,m\})$. Then $G:=\bigcup_{n=1}^{\infty}F_n$ is an Archimedean real $\Phi$-algebra. Moreover, if $f,g\in G$, then there exist $n\in\mathbb{N}$ such that $f,g\in F_n$. Since $F_n$ is a $\Phi$-algebra, we have $f^2+g^2\in F_n$ as well, and hence $\sqrt{f^2+g^2}\in F_{n+1}$. It follows that $G$ is square mean closed and so $G_\mathbb{C}$ is an Archimedean complex $\Phi$-algebra. Of course, the constant function $e$ with value 1 is invertible in $G_\mathbb{C}$. Let $r\in G_+$ be any invertible element and define
		\[
		z(x):=1+x\wedge{\textstyle\frac{1}{2}}r(x)\quad  (x \in [0,1]).
		\]
		Note that $z \in G$ satisfies $z\in\overset{\circ}{\Delta}(e,r)$. Moreover, since $r$ is positive, invertible, and continuous, there exists a $\delta>0$ such that $x<\frac{1}{2}r(x)$ for all $x\in[0,\delta)$. Thus $z(x)=1+x$ on $[0,\delta)$, and so $z$ is not invertible in $G$. Hence $G_\mathbb{C}^{-1}$ is not order open. 
	\end{remark}
	
	Our next lemma confirms that the map $z \mapsto z^{-1}$ on a uniformly complete $E$ indeed possesses an order derivative at every point in its domain $E^{-1}$.
	
	\begin{lemma}\label{L:derivative of x^{-1}}
		Suppose $E$ is uniformly complete. Define the function $f\colon E^{-1}\to E^{-1}$ by $f(z):=z^{-1}$. Then $f'(c):=-c^{-2}$ is the order derivative of $f$ at $c$ for every $c\in E^{-1}$.
	\end{lemma}
	
	\begin{proof}
		For $c\in E^{-1}$ we have $U_c:=\overset{\circ}{\Delta}(c,\frac{1}{2}|c|) \subseteq E^{-1}$ as seen in the proof of Lemma~\ref{L:locally ord cts}. Let $h_\alpha\to 0$ be such that $c+h_\alpha\in U_c$ for all $\alpha$. There exists a net $(p_\beta)_\beta$ with $p_\beta\downarrow 0$ such that for every $\beta$, there exists an $\alpha_0$ such that for all $\alpha\geq\alpha_0$ we have $|h_\alpha|\leq p_\beta$. Thus for each $\beta$ there is an $\alpha_0$ such that 
		\begin{align*}
			|f(c+h_\alpha)-f(c)+h_\alpha c^{-2}|&=|-c^{-2}(c+h_\alpha)^{-1}h_\alpha c +c^{-2}(c+h_\alpha)^{-1}h_\alpha(c+h_\alpha)|\\&=|h_\alpha|^2|c|^{-2}|c+h_\alpha|^{-1}\\&\le 2|h_\alpha||c|^{-3}p_\beta
		\end{align*}
		for all $\alpha\ge\alpha_0$. It follows that $f'(c)=-c^{-2}$ at every $c\in E^{-1}$.  
	\end{proof}
	
	The following theorem extends the four essential differentiablity theorems from classical complex analysis to the Archimedean complex $\Phi$-algebra setting. The proof of these rules are slight variations of standard arguments. We later present a variation of the chain rule, see Theorem~\ref{T: super Diff rules}($iii$), where the use of $D$-neighborhoods is not needed in the assumption.
	
	\begin{theorem}\label{T: Diff rules}
		Let $f\colon \mathrm{dom}(f)\to E$ and $g\colon \mathrm{dom}(g)\to E$ be functions. The following hold. 
		\begin{itemize}
			\item[$(i)$](sum rule) If $c\in\mathrm{dom}(f)\cap\mathrm{dom}(g)$ and $f$ and $g$ have order derivatives $f'(c)$ and $g'(c)$ at $c$, respectively, then $f'(c)+g'(c)$ is the order derivative of the sum map $z\mapsto f(z)+g(z)$ at $c$.\vskip .2 cm
			
			\item[$(ii)$](product rule) If $c\in\mathrm{dom}(f)\cap\mathrm{dom}(g)$ and $f$ and $g$ have order derivatives $f'(c)$ and $g'(c)$ at $c$, respectively, then $f'(c)g(c)+f(c)g'(c)$ is the order derivative of the product map $z\mapsto f(z)g(z)$ at $c$. \vskip .2 cm
			
			\item[$(iii)$](chain rule) Suppose that $f$ and $g$ are functions so that $\mathrm{range}(g) \subseteq \mathrm{dom}(f)$. Assume that $g$ has order derivative $g'(c)$ at $c\in\mathrm{dom}(g)$ and $f$ has order derivative $f'(g(c))$ at $g(c)$. If there exist $D$-neighborhoods $U_c$ for $g$ at $c$ and $V_{g(c)}$ for $f$ at $g(c)$ such that $g(U_c)\subseteq V_{g(c)}$, then $f'(g(c))g'(c)$ is the order derivative of the composite map $z\mapsto f(g(z))$ at $c$. \vskip .2 cm
			
			\item[$(iv)$](quotient rule) Suppose $c\in\mathrm{dom}(f)\cap\mathrm{dom}(g)$ and $f$ and $g$ have order derivatives $f'(c)$ and $g'(c)$ at $c$, respectively, and $\mathrm{range}(g) \subseteq E^{-1}$. If $E$ is uniformly complete and there exists $D$-neighborhoods $U_c$ for $g$ at $c$ and $V_{g(c)}$ for $z\mapsto z^{-1}$ at $g(c)$ such that $g(U_c)\subseteq V_{g(c)}$, then $\Bigl(f'(c)g(c)-f(c)g'(c)\Bigr)g(c)^{-2}$ is the order derivative of the quotient map $z\mapsto f(z)g(z)^{-1}$ at $c$.
		\end{itemize}
	\end{theorem}
	
	\begin{proof}
		$(i)$ If $\overset{\circ}{\Delta}(c,r_1) \subseteq \mathrm{dom}(f)$ and $\overset{\circ}{\Delta}(c,r_2) \subseteq \mathrm{dom}(g)$ with $r_1$ and $r_2$ positive and invertible, then by \cite[Theorem~142.2(iii)]{Zaanen2} $r:=r_1 \wedge r_2$ is a positive invertible element, and we have $U_c:=\overset{\circ}{\Delta}(c,r) \subseteq \mathrm{dom}(f+g)=\mathrm{dom}(f)\cap \mathrm{dom}(g)$. Let $h_\alpha\to 0$ be such that $c+h_\alpha\in U_c$ for all $\alpha$. If $q_\beta\downarrow 0$ and $p_\gamma\downarrow 0$ are so that for fixed $\beta$ and $\gamma$ there is an $\alpha_0$ for which
		\[
		\bigl|f(c+h_\alpha)-f(c)-h_\alpha f'(c)\bigr|\le |h_\alpha|q_\beta\quad\mbox{and}\quad\bigl|g(c+h_\alpha)-g(c)-h_\alpha g'(c)\bigr|\le |h_\alpha|p_\gamma\qquad(\alpha\geq\alpha_0),
		\]
		then 
		\[
		\bigl|(f+g)(c+h_\alpha)-(f+g)(c)-h_\alpha (f'(c)+g'(c))\bigr|\le|h_\alpha|(q_\beta+p_\gamma).
		\]
		Hence $f'(c)+g'(c)$ is the order derivative of $z\mapsto f(z)+g(z)$ at $c$.
		
		$(ii)$ Since $\mathrm{dom}(fg)=\mathrm{dom}(f)\cap \mathrm{dom}(g)$, there is a positive invertible element $r$ such that, by the same argument as in $(i)$, we have $U_c:=\overset{\circ}{\Delta}(c,r) \subseteq \mathrm{dom}(fg)$. Let $h_\alpha\to 0$ be such that $c+h_\alpha\in U_c$ for all $\alpha$. As $h_\alpha\to 0$, there is a sequence $q_\beta\downarrow 0$ such that for any $\beta$ there is an $\alpha_0$ such that $|h_\alpha|\le q_\beta$ for all $\alpha\geq\alpha_0$. Let $p_\gamma\downarrow 0$ and $l_\lambda\downarrow 0$ be so that for fixed $\gamma$ and $\lambda$ there is an $\alpha_1$ such that
		\[
		\bigl|f(c+h_\alpha)-f(c)-h_\alpha f'(c)\bigr|\le |h_\alpha|p_\gamma\quad\mbox{and}\quad\bigl|g(c+h_\alpha)-g(c)-h_\alpha g'(c)\bigr|\le |h_\alpha|l_\lambda\qquad(\alpha\geq \alpha_1).
		\]
		Net let $\alpha_2\geq\alpha, \alpha_1$. By adding and subtracting $f(c)g(c+h_\alpha)$ and
		\[
		\Bigl(f(c+h_\alpha)-f(c)\Bigr)g(c)
		\]
		inside the modulus in the first line below, we obtain
		\begin{align*}
			\Bigl|f(c+h_\alpha)g(c+h_\alpha)-&f(c)g(c)-h_\alpha\Bigl(f'(c)g(c)+f(c)g'(c)\Bigr)\Bigr|\\&\le
			\Bigl|f(c+h_\alpha)-f(c)\Bigr|\Bigl|g(c+h_\alpha)-g(c)\Bigr|+|h_\alpha|p_\gamma|g(c)|+|h_\alpha|l_\lambda|f(c)|\\&\le|h_\alpha|\bigl((p_m\gamma+|f'(c)|)(|h_\alpha|l_\gamma+|h_\alpha||g'(c)|)+p_\gamma|g(c)|+l_\lambda|f(c)|\bigr)\\&\le|h_\alpha|\Bigl((p_\gamma+|f'(c)|)(q_\beta(l_\gamma+|g'(c)|)+p_\gamma|g(c)|+l_\lambda|f(c)|\Bigr)\\&=|h_\alpha|r_{\beta,\gamma,\lambda}
		\end{align*}
		whenever $\alpha\geq\alpha_2$, and where $r_{\beta,\gamma,\lambda}\downarrow 0$ is the net 
		\[
		r_{\beta,\gamma,\lambda}:=(p_\gamma+|f'(c)|)(q_\beta(l_\lambda+g'(c)))+p_\gamma|g(c)|+l_\lambda|f(c)|.
		\]
		Hence $f'(c)g(c)+f(c)g'(c)$ is the order derivative of $z\mapsto f(z)g(z)$ at $c$.
		
		$(iii)$ Let $U_c:=\overset{\circ}{\Delta}(c,r)$ be a $D$-neighborhood for $g$ at $c$ and $V_{g(c)}$ be a $D$-neighborhood for $f$ at $g(c)$ such that $g(U_c)\subseteq V_{g(c)}$. Then $U_c \subseteq \mathrm{dom}(g)\cap g^{-1}(\mathrm{dom}(f))=\mathrm{dom}(f\circ g)$. Let $h_\alpha\to 0$ be so that $c+h_\alpha\in U_c$ for every $\alpha$. Suppose that $q_\beta\downarrow 0$ regulates the order differentiability of $g$ at $c$; that is for fixed $\beta$ there exists an $\alpha_{0}$ for which 
		\[
		\Bigl|g(c+h_\alpha)-g(c)-h_\alpha g'(c)\Bigr|\le |h_\alpha|q_\beta\qquad(\alpha\geq\alpha_0).
		\]
		Similarly, assume $p_\gamma\downarrow 0$ regulates the order differentiability of $f$ at $g(c)$. Then using the order continuity of $g|_{U_c}$ at $c$ from Lemma~\ref{L:diff is cont} and the assumption that $g(U_c)\subseteq V_{g(c)}$, we obtain
		\[
		\Bigl|f\bigl(g(c)+\bigl(g(c+h_\alpha)-g(c)\bigr)\bigr)-f(g(c))-\bigl(g(c+h_\alpha)-g(c)\bigr)f'(g(c))\Bigr|\le|g(c+h_\alpha)-g(c)|p_\gamma
		\]
		for $\alpha\geq\alpha_0$. By adding in and subtracting $f'(g(c))g(c+h_\alpha)$ and $f'(g(c))g(c)$ below, it follows that 
		\begin{align*}
			|f(g(c+h_\alpha))-f(g(c))-h_\alpha f'(g(c))g'(c)|&\le |g(c+h_\alpha)-g(c)|p_\gamma+|h_\alpha||f'(g(c))|q_\beta\\&\le|h_\alpha|(q_\beta p_\gamma+|g'(c)|p_\gamma+|f'(g(c))|q_\beta)\\&=|h_\alpha|l_{\beta,\gamma},
		\end{align*}
		for all $\alpha\geq\alpha_0$, and where $l_{\beta,\gamma}\downarrow 0$ is the net
		\[
		l_{\beta,\gamma}:=q_\beta p_\gamma+|g'(c)|p_\gamma+|f'(g(c))|q_\beta.
		\]
		Hence $f'(g(c))g'(c)$ is the order derivative of $z\mapsto f(g(z))$ at $c$.
		
		$(iv)$ Suppose $E$ is uniformly complete. Assume there exist $D$-neighborhoods $U_c$ for $g$ at $c$ and $V_{g(c)}$ for the map $z\mapsto z^{-1}$ at $g(c)$ such that $g(U_c) \subseteq V_{g(c)}$. There is an invertible element $r \in E_+$ such that $U_c:=\overset{\circ}{\Delta}(c,r) \subseteq \mathrm{dom}(f)\cap \mathrm{dom}(g)=\mathrm{dom}(f/g)$, by the same argument used in part $(i)$ and $(ii)$. Let $h_\alpha\to 0$ be so that $c+h_\alpha\in U_c$ for every $\alpha$. By Lemma~\ref{L:derivative of x^{-1}} and the chain rule $(iii)$ it follows that $-g'(c) g(c)^{-2}$ is the order derivative for $z\mapsto g(z)^{-1}$ at $c$, and by the product rule $(ii)$ we have that
		\[
		f'(c) g(c)^{-1}-f(c)g'(c) g(c)^{-2}=\Bigl(f'(c)g(c)-f(c)g'(c)\Bigr)g(c)^{-2}
		\]
		is the order derivative of $z\mapsto f(z)g(z)^{-1}$ at $c$.
	\end{proof}
	
	We conclude this section with the definition of the order theoretical analogue of a holomorphic function. 
	
	\begin{definition}
		Let $U\subseteq E$ be an order open set. A function $f \colon U \to E$ is said to be \emph{holomorphic} on $U$, if $f$ is order differentiable at every $c \in U$.
	\end{definition}
	
	An example of a holomorphic function is the map $z \mapsto z^{-1}$ on $E^{-1}$ for uniformly complete $E$ by Lemmas~\ref{L:locally ord cts}\&\ref{L:derivative of x^{-1}}.
	
	\section{Super and $\sigma$-super order differentiability}\label{S:super}
	
	We introduce the concepts of super and $\sigma$-super order differentiability in this section as a means to strengthen the conclusions of Lemma~\ref{L:locally ord cts} and Theorem~\ref{T: Diff rules}.
	
	\begin{definition}\label{D:super order derivative}
		We call a function $f\colon\mathrm{dom}(f)\to E$   \emph{super order differentiable at $c\in\mathrm{dom}(f)$} if $f$ is order differentiable at $c$ and the inequality \eqref{E:order derivative} in Definition~\ref{D:order derivative} holds for all nets $h_\alpha\to 0$ such that $c+h_\alpha\in\mathrm{dom}(f)$ for all $\alpha$.
	\end{definition}
	
	By following the proof of Lemma~\ref{L:diff is cont}, one readily deduces the lemma below.
	
	\begin{lemma}\label{L:super diff is ord cont}
		If a function $f\colon \mathrm{dom}(f)\to E$ is super order differentiable at $c\in\mathrm{dom}(f)$, then $f$ is order continuous at $c$.
	\end{lemma}
	
	Next we provide an analogue of Theorem~\ref{T: Diff rules}$(i)-(iii)$ for super order differentiable functions. Notice that, unlike Theorem~\ref{T: Diff rules}$(iii)$, no assumptions regarding the $D$-neighborhoods of $g$ and $f$ are required in the super chain rule below. The proof is similar to that of Theorem ~\ref{T: Diff rules}$(i)-(iii)$.
	
	\begin{theorem}\label{T: super Diff rules}
		Let $f\colon\mathrm{dom}(f)\to E$ and $g\colon \mathrm{dom}(g)\to E$ be functions. The following hold. 
		\begin{itemize}
			\item[$(i)$](super sum rule) If $c\in\mathrm{dom}(f)\cap\mathrm{dom}(g)$ and $f$ and $g$ are super order differentiable at $c$, then $f+g$ is super order differentiable at $c$. \vskip .2 cm
			
			\item[$(ii)$](super product rule) If $c\in\mathrm{dom}(f)\cap\mathrm{dom}(g)$ and $f$ and $g$ are super order differentiable at $c$, then $fg$ is super order differentiable at $c$. \vskip .2 cm
			
			\item[$(iii)$](super chain rule) Suppose that $f$ and $g$ are functions so that $\mathrm{range}(g) \subseteq \mathrm{dom}(f)$. If $g$ is order differentiable at $c$ and $f$ is super order differentiable at $c$, then $f\circ g$ is order differentiable at $c$. If in addition $g$ is super order differentiable at $c$, then $f\circ g$ is super order differentiable at $c$. \vskip .2 cm
		\end{itemize}
	\end{theorem}
	
	At this point, one might suspect that the map $z\mapsto f(z)g(z)^{-1}$ is super order differentiable whenever $f$ and $g$ are. However, this is not the case in general. In fact, the map $z\mapsto f(z)g(z)^{-1}$ can even fail to be order continuous.
	
	\begin{remark}\label{R:inverse nets don't converge} It is not true in general that the map $z\mapsto z^{-1}$ on an Archimedean complex $\Phi$-algebra is order continuous. To demonstrate this fact, recall from Remark~\ref{R:conv iff bdd} that a net $(z_\alpha)_\alpha$ of positive invertible elements converging in order to a positive invertible element $z$ has the property that $z_\alpha^{-1}\to z^{-1}$ if and only if there exists an index $\alpha_0$ such that $\{|z_\alpha|^{-1}\colon \alpha\ge \alpha_0\}$ is order bounded. 
		With this logical equivalence in mind, consider the space $\mathbb{C}^\mathbb{N}$ of complex valued sequences. For $(k,l)\in\mathbb{N}\times\mathbb{N}$ define $f_{k,l}\in \mathbb{C}^\mathbb{N}$ by
		\[
		f_{k,l}(n):=\begin{cases}1&\mbox{if $1\le n\le k$}\\l^{-1}&\mbox{for all $n>k$}\end{cases},
		\]
		and by ordering the elements of $\mathbb{N}\times\mathbb{N}$ coordinatewise, we obtain the net $(f_{k,l})_{(k,l)}$. Furthermore, for $k\in\mathbb{N}$ define $g_k\in\mathbb{C}^\mathbb{N}$ by
		\[
		g_k(n):=\begin{cases}0&\mbox{if $1\le n\le k$}\\1&\mbox{for all $n>k$}\end{cases},
		\]
		and note that for the constant 1 sequence $e$ we have $|f_{k,l}-e|\le g_k$ for all $l\in\mathbb{N}$, so $f_{k,l}\to e$ as $g_k\downarrow 0$. It follows that there is no $(k_0,l_0)\in \mathbb{N}\times\mathbb{N}$ such that $\bigl\{f_{k,l}^{-1}\colon (k,l)\ge(k_0,l_0)\bigr\}$ is order bounded, since the subset $\bigl\{f_{k_0,l}^{-1}\colon l\ge l_0\bigr\}$ consists of sequences whose values on $n>k_0$ are unbounded. Hence $f_{k,l}^{-1}\not\to e$ by what was discussed in the first part of this remark.
	\end{remark} 
	
	Despite the lack of order continuity for the map $z\mapsto f(z)g(z)^{-1}$ in general, we prove in Proposition~\ref{P: sigma-continuity of z^{-1}} that this map, when defined on a universally complete complex vector lattice, is always $\sigma$-order continuous. We stress here that universally complete complex vector lattices are Archimedean complex $\Phi$-algebras, as noted in \cite{RoeSch}.
	
	It is readily checked that any band in a universally complete complex vector lattice $E$ is itself a universally complete complex vector lattice.
	
	\begin{notation}
		Suppose $E$ is universally complete. Given $z\in E$, we denote by $z^\ast$ the multiplicative inverse of $z$ in the band $B_z$ generated by $z$ in $E$.
	\end{notation}
	
	\begin{proposition}\label{P: sigma-continuity of z^{-1}}
		If $E$ is universally complete, then the function $f\colon E^{-1}\to E^{-1}$ defined by $f(z):=z^{-1}$ is $\sigma$-order continuous.
	\end{proposition}
	
	\begin{proof}
		Let $c\in E^{-1}$, and let $(z_n)_{n\geq 0}$ be a sequence in $E^{-1}$ such that $z_n\to c$. Then there exists a sequence $(p_m)_{m\geq 0}$ with $p_m\downarrow 0$ such that for every $m\ge 0$, there exists an $N\ge 0$ such that for all $n\geq N$ we have $|z_n-c|\leq p_m$. For each $m\ge 1$, let $B_m$ denote the principal band generated by $(|c|-p_m)^+$, and set $B_0=\{0\}$. Furthermore, let $\mathbb{P}_m$ be the band projection onto $B_m$, and set $\mathbb{Q}_m=\mathbb{P}_m-\mathbb{P}_{m-1}$. Then for $m<n$ we have
		\[
		0\leq\mathbb{Q}_m(e)\wedge\mathbb{Q}_n(e)=(\mathbb{P}_m(e)-\mathbb{P}_{m-1}(e))\wedge(\mathbb{P}_n(e)-\mathbb{P}_{n-1}(e))\leq\mathbb{P}_{n-1}(e)\wedge(e-\mathbb{P}_{n-1}(e))=0,
		\]
		so $(\mathbb{Q}_m)_{m\ge 1}$ is a pairwise disjoint sequence of band projections.
		
		Let $m\in\mathbb{N}$ be arbitrary. Let $N$ be such that $|z_n-c|\leq p_m$ holds for all $n\geq N$. It follows that
		\[
		(|c|-p_m)^+\leq|z_n|
		\]
		holds for all $n\geq N$, and thus
		\[
		\mathbb{Q}_m|z_n^{-1}|\leq\mathbb{Q}_m((|c|-p_m)^+)^\ast
		\]
		holds for all $n\geq N$. If we define $u_m:=\left(\bigvee_{k\leq N}\mathbb{Q}_m|z_k^{-1}|\right)\vee\mathbb{Q}_m((|c|-p_m)^+)^\ast$ for each $m\in\mathbb{N}$, it follows that $\mathbb{Q}_m|z_n^{-1}|\leq u_m$ for all $n\geq 0$. Since $(u_m)_{m\geq 1}$ is a pairwise disjoint sequence, $u:=\sup_{m\ge 1}u_m$ exists in $E_+$ as $E$ is universally complete. For every $m\in\mathbb{N}$ and every $n \in \mathbb{N}_0$ we have that $\mathbb{Q}_m|z_n^{-1}|\leq u$, thus 
		\begin{align*}
			|z_n|^{-1}=\lim_{M\to\infty}\mathbb{P}_M|z_n|^{-1}=\lim_{M\to\infty}\sum_{m\leq M}\mathbb{Q}_m|z_n|^{-1}=\lim_{M\to\infty}\sup_{m\leq M}\mathbb{Q}_m|z_n|^{-1}\leq u.
		\end{align*}
		for each $n\in\mathbb{N}_{0}$. Finally, observe that
		\[
		|z_n^{-1}-c^{-1}|=|z_n-c||c^{-1}||z_n^{-1}|\leq|z_n-c||c^{-1}|u,
		\]
		proving that $z_n^{-1}\to c^{-1}$.
	\end{proof}
	
	The assumption of universal completeness in Proposition~\ref{P: sigma-continuity of z^{-1}} cannot be relaxed, as the following remark illustrates.
	
	\begin{remark}\label{uc is needed}
		Consider the Dedekind complete complex $\Phi$-algebra $\ell^\infty_\mathbb{C}$. We look at the sequence $(z_k)_{k\geq 0}$ defined by $z_k := f_{k,k}$ as in Remark~\ref{R:inverse nets don't converge}. Then $z_k \to e$ and $z_k^{-1} \in \ell^\infty_\mathbb{C}$ for every $k \in \mathbb{N}_0$, but $z_k^{-1} \not\to e$ as no tail of $(z_k^{-1})_{k\ge 0}$ is order bounded. Hence the map $z\mapsto z^{-1}$ is not $\sigma$-order continuous on the invertible elements of $\ell^\infty_\mathbb{C}$.
	\end{remark}
	
	The $\sigma$-order continuity of the map $z\mapsto z^{-1}$ on a universally complete space
	motivates the following definition.
	
	\begin{definition}\label{D:sigma super ord diff}
		We call a function $f\colon\mathrm{dom}(f)\to E$   $\sigma$-\emph{super order differentiable at $c$} if $f$ is order differentiable at $c$ and the inequality \eqref{E:order derivative} in Definition~\ref{D:order derivative} holds for all sequences $h_n\to 0$ such that $c+h_n\in\mathrm{dom}(f)$ for all $n\geq 0$.
	\end{definition}
	
	The notion of $\sigma$-super order differentiability allows us to recover a super quotient rule. We begin by proving that the map $z\to z^{-1}$ is $\sigma$-super order differentiable on universally complete complex vector lattices.
	
	\begin{lemma}\label{L:super derivative of x^{-1}} Let $E$ be universally complete, and define the function $f\colon E^{-1}\to E^{-1}$ by $f(z):=z^{-1}$. Then $f$ is $\sigma$-super order differentiable at $c$ for every $c\in E^{-1}$.
	\end{lemma}
	
	\begin{proof}
		For $c\in E^{-1}$, note that $\overset{\circ}{\Delta}(c,\frac{1}{2}|c|) \subseteq E^{-1}$. Indeed, if $z \in E$ such that $|z-c|\ll \frac{1}{2}|c|$, it follows that $\frac{1}{2}|c| \le |z|$. Hence $z \in E^{-1}$. Let $h_n\to 0$ be such that $c+h_n\in E^{-1}$ for all $n\in\mathbb{N}_{0}$. Similarly to what was shown in Proposition~\ref{P: sigma-continuity of z^{-1}}, we have that $u:=\sup_{n\ge 0}|c+h_n|^{-1}$ exists in $E_+$ as $E$ is universally complete. There exists a sequence $(p_m)_{m\geq 0}$ with $p_m\downarrow 0$ such that for every $m\ge 0$, there exists an $N\ge 0$ such that for all $n\geq N$ we have $|h_n|\leq p_m$. Thus for $m\ge 0$ there is an $N\ge 0$ such that 
		\begin{align*}
			|f(c+h_n)-f(c)+h_n c^{-2}|&=|-c^{-2}(c+h_n)^{-1}h_n c +c^{-2}(c+h_n)^{-1}h_n(c+h_n)|\\&=|h_n|^2|c|^{-2}|c+h_n|^{-1}\\&\le |h_n||c|^{-2}u p_m
		\end{align*}
		for all $n\ge N$. It follows that $f$ is $\sigma$-super order differentiable at $c$.  \end{proof}
	
	\begin{theorem}\label{T: super quotient rule}(super quotient rule)
		Let $E$ be universally complete and let $f\colon\mathrm{dom}(f)\to E$ and $g\colon \mathrm{dom}(g)\to E$ be functions. Suppose $c\in\mathrm{dom}(f)\cap\mathrm{dom}(g)$ and $f$ and $g$ are $\sigma$-super order differentiable at $c$ and $\mathrm{range}(g) \subseteq E^{-1}$. Then the quotient map $z\mapsto f(z)g(z)^{-1}$ is $\sigma$-super order differentiable at $c$.
	\end{theorem}
	
	\begin{proof}
		Let $h_n\to 0$ be so that $c+h_n\in\mathrm{dom}(f/g)$ for every $n\in\mathbb{N}_{0}$. By Lemma~\ref{L:super derivative of x^{-1}} and the super chain rule Theorem~\ref{T: super Diff rules}$(iii)$ it follows that the function $z\mapsto g(z)^{-1}$ is  $\sigma$-super order differentiable at $c$. It then follows from the super product rule Theorem~\ref{T: super Diff rules}$(ii)$ that $z\mapsto f(z)g(z)^{-1}$ is $\sigma$-super order differentiable at $c$.
	\end{proof}
	
	\section{improved Cauchy-Hadamard formulas}\label{S:CHF}
	
	We improve our Cauchy-Hadamard formulas from \cite{RoeSch} in this section and therefore focus solely on universally complete spaces, as we did in \cite{RoeSch}.
	
	\begin{notation}
		$E$ will denote a universally complete complex vector lattice in this section (with corresponding real part $F$ as already designated).
	\end{notation}
	
	We next provide some useful basic properties of $E$, which we use freely throughout the rest of the paper. Note that (ii) follows from (i), which is a consequence of \cite[Theorem~2.37]{AlipBurk}. Together, \cite[Theorem~2.44]{AlipBurk} and \cite[Theorem~18.13]{Zaanen1} imply (iii)--(v). The veracity of (vi) is explained in \cite[Remark~3.3]{RoeSch}.
	
	For all $z,w\in E$ and any order projection $\P$ on $E$,
	
	\begin{itemize}
		\item[(i)] $\P(zw)=z\P(w)$.
		\item[(ii)] $\P(zw)=\P(z)\P(w)$.
		\item[(iii)] $\mathbb{P}(|z|)=|\mathbb{P}(z)|$, that is, $\mathbb{P}$ is a complex vector lattice homomorphism.
		\item[(iv)] $\mathbb{P}(\sup A)=\sup[\mathbb{P}(A)]$ for all $\varnothing\neq A\subseteq F$ bounded above.
		\item[(v)] $\P$ is order continuous.
		\item[(vi)] $z$ is invertible if and only if $|z|$ is a weak order unit. 
	\end{itemize}
	
	In particular, $E$ is uniformly complete. Thus for every $x\in E_+$, and $n\in\N$, there exists a unique $y\in E_+$ such that $y^n=x$ (see \cite[Corollary~6]{BeuHuij}). This $n$th root of $x$ is denoted by $x^{1/n}$ in the remainder of this manuscript.
	
	\begin{remark}
		For sequences, our definition of order convergence seems different at first glance from the definition given in \cite[Ch.~4 S.~10]{littleZaanen} and used in \cite{RoeSch}, but it is in fact, equivalent if $E$ is Dedekind complete. Indeed, if $q_m \downarrow 0$ is such that for all $m\ge 0$ there exists $N\ge 0$ so that $|z_n - z|\le q_m$ whenever $n\ge N$, then for all $m\ge 0$, define $p_m:=\sup_{n\ge m}|z_n-z|$, which are well-defined elements in $E$ by Dedekind completeness. Clearly, the sequence $(p_m)_{m\ge 0}$ is downward directed and satisfies $|z_m-z|\le p_m$ for all $m\ge 0$. By Dedekind completeness once more, let $p := \inf_{m\ge 0}p_m$. For any $k \ge 0$ there is an $m_k$ such that $p_{m_k}\le q_k$ and so, we have $p \le \inf_{k\ge 0}p_{m_k} \le \inf_{k \ge 0 }q_k = 0$, showing that $p_m\downarrow 0$. Note that the converse implication clearly holds.
	\end{remark}
	
	For an order bounded sequence $(x_n)_{n\geq 0}$ in $F$, we as usual write
	\[
	\limsup_{n\to\infty}x_n:=\inf_{n\geq 0}\sup_{m\geq n}x_n\qquad \text{and}\qquad \liminf_{n\to\infty}x_n:=\sup_{n\geq 0}\inf_{m\geq n}x_n.
	\]
	Observe that $\limsup_{n\to\infty}x_n$ and $\liminf_{n\to\infty}x_n$ are the order limits of the sequences 
	\[
	\left(\sup_{m\geq n}x_n\right)_{n\geq 0}\qquad \textnormal{and}\qquad \left(\inf_{m\geq n}x_n\right)_{n\geq 0}, 
	\]
	respectively.
	
	We proceed with some notes on the sup-completion $F^s$ of $F$, a notion introduced in \cite[Section 1]{Donner}. Using the terminology from \cite[Definition 1.1]{Donner}, a \textit{cone} $(C,+)$ is a commutative monoid, that is closed under a nonnegative real number scalar multiplication, for which the following hold:
	
	\begin{itemize}
		\item[(i)] $\alpha(x+y)=\alpha x+\alpha y\quad (\alpha\in[0,\infty),\ x,y\in C)$,
		\item[(ii)] $(\alpha+\beta)x=\alpha x+\beta x\quad (\alpha,\beta\in[0,\infty),\ x\in C)$,
		\item[(iii)] $\alpha(\beta x)=(\alpha\beta)x\quad (\alpha,\beta\in[0,\infty),\ x\in C)$,
		\item[(iv)] $1x=x\quad (x\in C)$, and
		\item[(v)] $0x=0\quad (x\in C)$.
	\end{itemize}
	In (v) above, the first $0$ denotes the real number zero, while the second $0$ designates the identity element of $(C,+)$.
	
	Following \cite{Donner}, we will from now on denote a cone $(C,+)$ by $C$, for short. Given a cone $C$, we denote as in \cite{Donner} the set
	\[
	C_0:=\{x\in C\colon x\ \textnormal{possesses an inverse under}\ +\}.
	\]
	The nonnegative real number scalar multiplication restricted to $C_0$ can be extended to all of the real numbers by defining
	\[
	\alpha x:=-\alpha(-x)\quad \bigl(\alpha\in(-\infty,0),\ x\in C_0\bigr).
	\]
	It is readily checked that $C_0$ is a vector space under $+$ and this expanded scalar multiplication.
	
	A cone $C$ possessing a partial ordering $\leq$ is called an \textit{ordered cone} if the following hold:
	
	\begin{itemize}
		\item[(vi)] $x\leq y$ implies $x+z\leq y+z\quad (x,y,z\in C)$, and
		\item[(vii)] $x\leq y$ implies $\alpha x\leq \alpha y\quad \bigl(\alpha\in[0,\infty),\ x,y\in C\bigr)$.
	\end{itemize}
	An ordered cone that is a lattice with respect to its partial ordering is called a \textit{lattice cone}. A lattice cone $C$ is called \textit{Dedekind complete} if every nonempty subset of $C$ which is bounded above (respectively, bounded below) possesses a supremum (respectively, infimum) in $C$.
	
	\begin{proposition}\label{P:sup comp}\cite[Theorem~1.4]{Donner}
		There exists an essentially unique cone $F^s$ (called the \textit{sup-completion} of $F$) for which the following hold.
		
		\begin{itemize}
			\item[$(i)$] $F^s$ is Dedekind complete.
			
			\item[$(ii)$] $F=(F^s)_0$ with coinciding algebraic and order structures.
			
			\item[$(iii)$] For each $y\in F^s$, we have
			\[
			y=\sup\{x\in F\colon  x\leq y\}.
			\]
			
			\item[$(iv)$] $z+(x\wedge y)=(z+x)\wedge(z+y)\quad (x\in F,\ y,z\in F^s)$.
			
			\item[$(v)$] $\{y\in F^s\colon \textnormal{there exists}\ x\in F\ \textnormal{such that}\ y\leq x\}\subseteq F$.
			
			\item[$(vi)$] $F^s$ has a largest element.
			
			\item[$(vii)$] For any $\varnothing\neq A,B\subseteq F^s$ for which $\sup A=\sup B$, we have
			\[
			\underset{a\in A}{\sup}(a\wedge x)=\underset{b\in B}{\sup}(b\wedge x)
			\]
			for any $x\in F$.
		\end{itemize}
	\end{proposition}
	
	As noted in \cite[(P6)]{AzouNasri}, the following immediate consequence of Proposition~\ref{P:sup comp}$(vii)$ holds.
	
	\begin{proposition}\label{P: sup comp 2}
		For any $\varnothing\neq A\subseteq F^s$ and $x\in F$, we have
		\[
		\underset{a\in A}{\sup}(a\wedge x)=(\sup A)\wedge x.
		\]
	\end{proposition}
	
	A salient observation for our purposes is that Proposition~\ref{P:sup comp}$(i)$\&$(vi)$ together imply that $\sup A$ exists in $F^s$ for any nonempty $A\subseteq F^s$.
	
	We utilize the following proposition from \cite[Theorem~1.4]{Donner}.
	
	\begin{proposition}\label{P:dist lattice}
		$F^s$ is a distributive lattice.
	\end{proposition}
	
	Next we begin to enhance our Cauchy-Hadamard formulas originally given in \cite{RoeSch} to be able to consider unbounded spectra of convergence. For this task, we rely on the sup-completion of $F$. By upgrading our Cauchy-Hadarmard formulas, we are able to prove that analytic functions on $E$ are holomorphic. We begin with some relevant definitions. For more information on series and power series on universally complete complex vector lattices, we refer the reader to \cite{RoeSch}.
	
	\begin{definition}
		We say that a series $\sum_{n=0}^\infty a_n$ \textit{converges in order} if the sequence of partial sums $(\sum_{n=0}^ma_m)_{m\geq 0}$ converges in order in $E$. If $\sum_{n=0}^\infty|a_n|$ converges in order, then we say that $\sum_{n=0}^\infty a_n$ \textit{converges absolutely in order}.
	\end{definition}
	
	\begin{definition}
		A \textit{power series} on $E$, centered at $c\in E$, is of the form $\sum_{n=0}^\infty a_n(z-c)^n$, where $a_n\in E$ for every $n\in\mathbb{N}_{0}$ and $z$ is a variable in $E$. We say that a power series $\sum_{n=0}^\infty a_n(z-c)^n$ \textit{converges uniformly in order} on a region $D\subseteq E$ if there exists a sequence $p_m\downarrow 0$ such that for every $m\in\mathbb{N}_{0}$, there exists a $K\in\mathbb{N}_{0}$ such that
		\[
		\underset{z\in D}{\sup}\left|\sum_{n=0}^{k}a_n(z-c)^n-\sum_{n=0}^\infty a_n (z-c)^n\right|\leq p_m
		\]
		holds for every $k\geq K.$
	\end{definition}
	
	\begin{definition}
		Let $\sum_{n=0}^\infty a_n(z-c)^n$ be a power series. We define the \textit{spectrum of convergence} of the power series as 
		\[
		\Omega:=\left\{r\in E_+\colon \sum_{n=0}^\infty a_n(z-c)^n \text{ converges uniformly in order on}\ \bar{\Delta}(c,r) \right\}.
		\]
		We also define $\rho:=\sup\Omega$ to be the \emph{radius of convergence} of the power series, which exists in the sup-completion $F^s$ of $F$. 
	\end{definition}
	
	We next provide a series of remarks, definitions, and results which we use to improve the Cauchy-Hadarmard formulas given in \cite[Theorem~3.11]{RoeSch}.
	
	\begin{remark}
		By \cite[Proposition~3]{Azouzi}, any band projection $\mathbb{P}\colon F\to F$ can be extended to a left-order continuous, additive, and positively homogeneous map $\bar{\mathbb{P}}\colon F^s\to F^s$ via the formula
		\[
		\bar{\mathbb{P}}(y):=\sup\{\mathbb{P}(x)\colon x\in F,\ x\leq y\}\qquad (y\in F^s).
		\]
		Given $0\leq u\in F^s$, we define as in \cite{Azouzi} a map $\mathbb{P}_u\colon F\to F$ by
		\[
		\mathbb{P}_u(x):=\underset{n\ge 1}{\sup}\{x\wedge nu\}
		\]
		for all $x\in F_+$ and then
		\[
		\mathbb{P}_u(x):=\mathbb{P}_u(x^+)-\mathbb{P}_u(x^-)
		\]
		for all $x\in F$. By \cite[Lemma 4]{Azouzi}, the map $\mathbb{P}_u$ is a band projection on $F$; in fact, $\mathbb{P}_u=\mathbb{P}_{\mathbb{P}_u(e)}$. We denote the corresponding band $B_{\mathbb{P}_u(e)}$ by $B_u$, for short. Recall that $B_u$ is itself a (real) universally complete vector lattice.
	\end{remark}
	
	\begin{lemma}\label{L:B_u^s}
		Let $0\leq u\in F^s$. Then  $\bar{\mathbb{P}}_u(u)=u$ and $u\in B_u^s$.
	\end{lemma}
	
	\begin{proof}
		Let $(x_\alpha)_\alpha$ be a net in $F$ such that $x_\alpha\uparrow u$. From the left-order continuity of $\bar{\mathbb{P}}_u$ and Proposition~\ref{P:sup comp}$(iii)$ we have
		\begin{align*}
			\mathbb{P}_u(x_\alpha)=\bar{\mathbb{P}}_u(x_\alpha)\uparrow\bar{\mathbb{P}}_u(u)&=\sup\{\mathbb{P}_u(x)\colon x\in F,\ x\leq u\}\\
			&=\sup\Bigl\{\underset{n\ge 1}{\sup}\{x\wedge nu\}\colon x\in F,\ x\leq u\Bigr\}\\
			&=\sup\{x\colon x\in F,\ x\leq u\}\\
			&=u.
		\end{align*}
		Since $\mathbb{P}_u(x_\alpha)\in B_u$ for all $\alpha$, we conclude that $u\in B_u^s$.
	\end{proof}
	
	Next we provide an alternative mechanism for dividing positive elements of a sup-completion into finite and infinite parts to that found in \cite{AzouNasri}. As stated in the introduction, this new approach perhaps yields simpler formulas, which facilitate many of the proofs in this section.
	
	\begin{definition}
		For $0\leq u\in F^s$, we define
		\[
		u_\infty:=\underset{n\ge 1}{\inf} n^{-1}u
		\]
		and
		\[
		u_{\mathcal{F}}:=\sup\{x-x\wedge u_\infty:\ x\in F_+,\ x\leq u\}.
		\]
	\end{definition}
	
	Note that $x\wedge u_\infty\in F$ for every $x\in F$ by Proposition~\ref{P:sup comp}$(v)$. Ergo the element $u_{\mathcal{F}}$ is well-defined in $F^s$.
	
	The following lemma is required for our proof of Theorem~\ref{T:uF and u infty}. It is an immediate consequence of the definition of $u_\infty$.
	
	\begin{lemma}\label{L:muinfty}
		Given $0\leq u\in F^s$ and $m\in\mathbb{N}$, we have $mu_\infty=u_\infty$.
	\end{lemma}
	
	Our next result represents an alternative way of breaking positive elements of a sup-completion into disjoint finite and infinite parts to that found in \cite{AzouNasri}.
	
	\begin{theorem}\label{T:uF and u infty}
		Let $0\leq u\in F^s$. Then the following hold.
		
		\begin{itemize}
			\item[$(i)$] $u=u_{\mathcal{F}}+u_\infty$.
			
			\item[$(ii)$] $u_{\mathcal{F}}\wedge u_\infty=0$.
			
			\item[$(iii)$] $u_\infty=0$ or $u_\infty\in F^s\setminus F$.
			
			\item[$(iv)$] $u_{\mathcal{F}}\in F_+$.
		\end{itemize}
	\end{theorem}
	
	\begin{proof}
		$(i)$ Note that
		\begin{align*}
			u_{\mathcal{F}}+u_\infty&=\sup\{x-x\wedge u_\infty\colon x\in F_+, x\leq u\}+u_\infty\\
			&\geq\sup\{x+u_\infty-x\wedge u_\infty\colon x\in F_+, x\leq u\}\\
			&\geq\sup\{x\colon x\in F_+, x\leq u\}\\
			&=u,
		\end{align*}
		where the last equality follows from Proposition~\ref{P:sup comp}$(iii)$.
		
		Moreover, we have
		\begin{align*}
			u_{\mathcal{F}}+u_\infty&\leq u+u_\infty=u+\underset{n\ge 1}{\inf}\ n^{-1}u\leq\underset{n\ge 1}{\inf}\ (u+n^{-1}u)=\underset{n\ge 1}{\inf}\ (1+n^{-1})u=u.
		\end{align*}
		Thus $(i)$ holds.
		
		$(ii)$ If $x\in F_+$ and $x\leq u$, then by Proposition~\ref{P:sup comp}$(iv)$ and Lemma~\ref{L:muinfty}, we have
		\begin{align*}
			(x-x\wedge u_\infty)\wedge u_\infty&=x\wedge(u_\infty+x\wedge u_\infty)-x\wedge u_\infty\\
			&=x\wedge\Bigl((x+u_\infty)\wedge 2u_\infty\Bigr)-x\wedge u_\infty\\
			&=x\wedge\Bigl((x+u_\infty)\wedge u_\infty\Bigr)-x\wedge u_\infty\\
			&=x\wedge u_\infty-x\wedge u_\infty\\
			&=0.
		\end{align*}
		Then by Proposition~\ref{P: sup comp 2} and the string of equalities above, we have
		\begin{align*}
			0&\leq u_{\mathcal{F}}\wedge(u_\infty\wedge e)\\
			&=\bigl(\sup\{x-x\wedge u_\infty\colon x\in F_+,\ x\leq u\}\bigr)\wedge(u_\infty\wedge e)\\
			&=\sup\{(x-x\wedge u_\infty)\wedge(u_\infty\wedge e):\ x\in F_+,\ x\leq u\}\\
			&\leq\sup\{(x-x\wedge u_\infty)\wedge u_\infty:\ x\in F_+,\ x\leq u\}\\
			&=0.
		\end{align*}
		Thus $u_{\mathcal{F}}\wedge u_\infty\wedge e=0$. But since $e$ is a positive invertible element, we have $u_{\mathcal{F}}\wedge u_\infty=0$, as claimed.
		
		$(iii)$ Evident since $E$ is Archimedean.
		
		$(iv)$ We use proof by contradiction. To this end, suppose that $u_{\mathcal{F}}\in F^s\setminus F$. Then by part $(iii)$ of this theorem we have $(u_{\mathcal{F}})_\infty\in F^s\setminus F$, and hence $(u_{\mathcal{F}})_\infty>0$. Thus
		\begin{align*}
			u_{\mathcal{F}}\wedge u_\infty&\geq(u_{\mathcal{F}})_\infty\wedge u_\infty=\left(\underset{n\ge 1}{\inf}\ n^{-1}u_{\mathcal{F}}\right)\wedge\left(\underset{n\ge 1}{\inf}\ n^{-1}u\right)=\underset{n\ge 1}{\inf}\ n^{-1}u_{\mathcal{F}}=(u_{\mathcal{F}})_\infty > 0,
		\end{align*}
		a contradiction with part $(ii)$ of this theorem. Thus $u_{\mathcal{F}}\in F$ and hence $u_{\mathcal{F}}\in F_+$, which establishes $(iv)$.
	\end{proof}
	
	As an immediate consequence of Theorem~\ref{T:uF and u infty}, we have the following corollary.
	
	\begin{corollary}
		Let $0\leq u\in F^s$. The following are equivalent:
		\begin{itemize}
			\item[($i$)] $u\in F_+$,
			\item[($ii$)] $u_\infty=0$, and
			\item[($iii$)] $u=u_{\mathcal{F}}$.
		\end{itemize}
	\end{corollary}
	
	The finite and infinite parts of elements in the sup-completion, as was shown in Theorem~\ref{T:uF and u infty}, yields the corresponding threefold band decomposition of $F$. 
	
	\begin{theorem}\label{T:3parts}
		For $0\leq u\in F^s$, we have
		\[
		F= B_{u_{\mathcal{F}}}\oplus B_{u_\infty}\oplus B_u^d. 
		\]
	\end{theorem}
	
	\begin{proof}
		Let $0\leq u\in F^s$. We first show that $B_{u_{\mathcal{F}}}\cap B_{u_\infty}=\{0\}$. To this end, suppose $x\in B_{u_{\mathcal{F}}}\cap B_{u_\infty}$, with $x\geq 0$. Then $x=\mathbb{P}_{u_{\mathcal{F}}}(x)=\mathbb{P}_{u_\infty}(x)$. Using Lemma~\ref{L:muinfty}, Theorem~\ref{T:uF and u infty}$(ii)$, and Proposition~\ref{P: sup comp 2}, we obtain
		\begin{align*}
			0\leq x&=x\wedge x=\mathbb{P}_{u_{\mathcal{F}}}(x)\wedge\mathbb{P}_{u_\infty}(x)=\underset{n\ge 1}{\sup}\{x\wedge nu_{\mathcal{F}}\}\wedge\underset{n\ge 1}{\sup}\{x\wedge nu_\infty\}\\
			&=\underset{n\ge 1}{\sup}\{x\wedge nu_{\mathcal{F}}\}\wedge(x\wedge u_\infty)=\underset{n\ge 1}{\sup}\{(x\wedge nu_{\mathcal{F}})\wedge(x\wedge u_\infty)\}\\&\leq\underset{n\ge 1}{\sup}\{n(u_{\mathcal{F}}\wedge u_\infty)\}=0.
		\end{align*}
		Hence $x=0$. It follows that $B_{u_{\mathcal{F}}}\cap B_{u_\infty}=\{0\}$.
		
		We next show that $B_u=B_{u_{\mathcal{F}}}+B_{u_\infty}$, first handling the inclusion $B_{u_{\mathcal{F}}}+B_{u_\infty}\subseteq B_u$. To this end, let $x+y\in B_{u_{\mathcal{F}}}+B_{u_\infty}$ with $x,y\geq 0$. Then
		\begin{align*}
			x+y&=\mathbb{P}_{u_{\mathcal{F}}}(x)+\mathbb{P}_{u_\infty}(y)=\underset{n\ge 1}{\sup}\{x\wedge nu_{\mathcal{F}}\}+\underset{m\ge 1}{\sup}\{y\wedge mu_\infty\}\\
			&\leq\underset{n\ge 1}{\sup}\{x\wedge nu\}+\underset{m\ge 1}{\sup}\{y\wedge mu\}\leq 2\mathbb{P}_u(x\vee y),
		\end{align*}
		and so $x+y\in B_u$.
		
		For the reverse inclusion, note that, from Theorem~\ref{T:uF and u infty}$(i)$  and \cite[Lemma 1.3]{Donner}, we have
		\[
		u=u_{\mathcal{F}}+u_\infty=u_{\mathcal{F}}\vee u_\infty+u_{\mathcal{F}}\wedge u_\infty.
		\]
		It thus follows from Theorem~\ref{T:uF and u infty}$(ii)$ that $u=u_{\mathcal{F}}\vee u_\infty$. Hence for any $0\leq x\in B_u$, we have from Lemma~\ref{L:muinfty}, Theorem~\ref{T:uF and u infty}$(ii)$, Proposition~\ref{P:sup comp}$(ii)$, and Proposition~\ref{P:dist lattice} that
		\begin{align*}
			x&=\mathbb{P}_u(x)=\underset{n\ge 1}{\sup}\{x\wedge nu\}=\underset{n\ge 1}{\sup}\{x\wedge(nu_{\mathcal{F}}+nu_\infty)\}=\underset{n\ge 1}{\sup}\{x\wedge nu_{\mathcal{F}}+x\wedge nu_\infty\}\\&=\underset{n\ge 1}{\sup}\{x\wedge nu_{\mathcal{F}}+x\wedge u_\infty\}=\underset{n\ge 1}{\sup}\{x\wedge nu_{\mathcal{F}}\}+x\wedge u_\infty=\underset{n\ge 1}{\sup}\{x\wedge nu_{\mathcal{F}}\}+\underset{m\ge 1}{\sup}\{x\wedge mu_\infty\}\\&=\mathbb{P}_{u_{\mathcal{F}}}(x)+\mathbb{P}_{u_\infty}(x).
		\end{align*}
		It hence follows that $B_u=B_{u_{\mathcal{F}}}+ B_{u_\infty}$. We therefore obtain
		\[
		F=B_u\oplus B_u^d=B_{u_{\mathcal{F}}}\oplus B_{u_\infty}\oplus B_u^d.
		\]
	\end{proof}
	
	We are almost ready to present our new and improved Cauchy-Hadamard formulas but require some lemmas first. 
	\begin{lemma}\label{L:P_Fu=u_F}
		Let $0\leq u\in F^s$. Then $\bar{\mathbb{P}}_{u_{\mathcal{F}}}(u)=u_{\mathcal{F}}$.
	\end{lemma}
	
	\begin{proof}
		First note that $\bar{\mathbb{P}}_{u_{\mathcal{F}}}(u)\geq\bar{\mathbb{P}}_{u_{\mathcal{F}}}(u_{\mathcal{F}})=u_{\mathcal{F}}$ 
		by Lemma~\ref{L:B_u^s}. Moreover, using Proposition~\ref{P:sup comp}$(iv)$ and Theorem~\ref{T:uF and u infty}$(ii)$, we have
		\begin{align*}
			\bar{\mathbb{P}}_{u_{\mathcal{F}}}(u)&=\sup\left\{\mathbb{P}_{u_{\mathcal{F}}}(x):\ x\in F_+,\ x\leq u\right\}=\sup\left\{\underset{n\ge 1}{\sup}\{x\wedge nu_{\mathcal{F}}\}:\ x\in F_+,\ x\leq u\right\}\\
			&\leq\underset{n\ge 1}{\sup}\{u\wedge nu_{\mathcal{F}}\}=\underset{n\ge 1}{\sup}\{(u_{\mathcal{F}}+u_\infty)\wedge nu_{\mathcal{F}}\}=\underset{n\ge 1}{\sup}\{u_{\mathcal{F}}+u_\infty\wedge(n-1)u_{\mathcal{F}}\}\\
			&=u_{\mathcal{F}},
		\end{align*}
		which proves the lemma.
	\end{proof}
	
	As in \cite{Donner}, we set $y^+:=y\vee 0$ for $y\in F^s$.
	
	\begin{lemma}\label{L:largestelmt}
		If $0\leq u\in F^s$, then $u_\infty$ is the largest element of $B_{u_\infty}^s$.
	\end{lemma}
	
	\begin{proof}
		Let $y\in B_{u_\infty}^s$. Then using Lemma~\ref{L:muinfty} in the fourth equality below, we obtain
		\begin{align*}
			y&\leq y^+=\bar{\mathbb{P}}_{u_\infty}(y^+)
			=\sup \{\mathbb{P}_{u_\infty}(x):\ x\in F_+,\ x\leq y^+\}\\
			&=\sup\left\{\underset{n\ge 1}{\sup}\{x\wedge nu_\infty\}:\ x\in F_+,\ x\leq y^+\right\}=\sup\left\{x\wedge u_\infty:\ x\in F_+,\ x\leq y^+\right\}\leq u_\infty.
		\end{align*}
	\end{proof}
	
	We next make use of the multiplication on $F^s$ introduced by Azouzi and Nasri in \cite[Section 3.2]{Azouzi}. This extended multiplication is defined for $0\leq x,y\in F^s$ as
	\[
	xy:=\sup\{vw:\ v,w\in F_+,\ v\leq x,\ w\leq y\}.
	\]
	
	The following useful lemma from \cite{Azouzi} will also be applied in our proofs.
	
	\begin{lemma}\label{L:23}\textbf{\textnormal{(}\cite[Lemma~23]{Azouzi}\textnormal{)}} Let $(x_\alpha)_\alpha, (y_\beta)_\beta$ be nets in $F^s_+$ such that $x_\alpha\uparrow x$ and $y_\beta\uparrow y$ for $0\leq x,y\in F^s$. Then $x_\alpha y_\alpha\uparrow xy$.
	\end{lemma}
	
	As an immediate consequence of Lemma~\ref{L:23}, we have the following.
	
	\begin{lemma}\label{L:23 leq}
		Let $(y_\alpha)_\alpha$ be a net in $F^s_+$ such that $y_\alpha\uparrow y$ for $0\leq y\in F^s$. If $0\leq x,u\in F^s$ and $xy_\alpha\leq u$ holds for every $\alpha$, then $xy\leq u$.
	\end{lemma}
	
	In order to formulate a Cauchy-Hadamard formula which considers unbounded spectra of convergence, we need a notion of a generalized inverse for positive elements in the sup-completion. We would like to remind the reader that for $z \in E$ the notation $z^*$ refers to the multiplicative inverse of $z$ in the band $B_z$.
	
	\begin{definition}
		Let $x\in F^s_+$. The \emph{generalized inverse} of $x$ in $F^s_+$, denoted by $x^{-1}$, is defined by $x^{-1} :=(x_{\mathcal{F}})^\ast + \infty_x^d$,
		where $\infty_x^d$ denotes the largest element of $(B_x^d)^s$.
	\end{definition}
	
	\begin{theorem}[\textbf{Cauchy-Hadarmard}]\label{T:CHF}
		Let $\sum_{n=0}^\infty a_n(z-c)^n$ be a power series on $E$. Also let $L:=\limsup_{n\to\infty}|a_n|^{1/n}$, which exists in $F^s$. The following identities hold.
		\begin{itemize}
			\item[($i$)] $B_L^d=B_{\rho_\infty}$.
			\item[($ii$)] $B_{L_\infty}=B_{\rho}^d$.
			\item[($iii$)] $(L_{\mathcal{F}})^\ast=\rho_{\mathcal{F}}$.
		\end{itemize}
		In particular, we have that $L^{-1}=\rho$ for the generalised inverse of $L$ in $F^s$.
	\end{theorem}
	
	\begin{proof}
		
		Towards proving $(i)$, let $r_\alpha\uparrow\rho$, where $r_\alpha\in\Omega$ for every $\alpha$. Fixing $\alpha$, we have that $r_\alpha\in\Omega$ implies $\sum_{n=0}^\infty a_nr_\alpha^n$ converges in order. Thus by \cite[Theorem~3.6(ii)]{RoeSch}, we have
		\[
		\underset{n\to\infty}{\limsup}(|a_n|^{1/n}r_\alpha)\leq e.
		\]
		
		Let $(x_n)_{n\geq 0}$ be a sequence in $E_+$ and set $x:=\underset{n\to\infty}{\limsup}\ x_n$. Note that $x$ is well-defined in $F^s$, and assume that $y\in E_+$. It follows from Lemma~\ref{L:23} and \cite[Lemma~24$(iii)$]{AzouNasri}, which states that
		\[
		x(y\vee z)=(xy)\vee(xz)\qquad (0\leq x,y,z\in F^s),
		\]
		that
		\begin{equation}\label{E: limsup mult}
			xy=\underset{n\to\infty}{\limsup}(x_ny).
		\end{equation}
		We thus obtain
		\begin{align*}
			Lr_\alpha=\underset{n\to\infty}{\limsup}\left(|a_n|^{1/n}r_\alpha\right)
			&\leq e.
		\end{align*}
		It follows from Lemma~\ref{L:23 leq} that $L\rho\leq e$, and thus $L\rho_\infty\leq e$.
		Then by Lemma~\ref{L:muinfty}, we have $nL\rho_\infty\leq e$ for each $n\in\mathbb{N}$, hence $L\rho_\infty=0$ since $F$ is Archimedean. It follows that $L\wedge\rho_\infty=0$. Thus if $x\in B_{\rho_\infty}$, then by Lemma~\ref{L:largestelmt} we have
		\[
		|x|\wedge L\leq\rho_\infty\wedge L=0,
		\]
		so $x\in B_L^d$.
		
		Next let $0\leq r\in B_L^d$. Then $\underset{n\to\infty}{\limsup}(|a_n|^{1/n}mr)=mLr=0\ll e$ for all $m\in\mathbb{N}$, so by \cite[Theorem~3.6(i)]{RoeSch} we have $mr\in\Omega$ for each $m\in\mathbb{N}$. Then $mr\leq\rho$ and hence $r\leq\frac{1}{m}\rho$ for every $m\in\mathbb{N}$. Thus we get $r\leq\rho_\infty$. Hence $r\in B_{\rho_\infty}$ and $(i)$ holds.
		
		We proceed with the proof of $(ii)$. For this task, let $r\in\Omega$. Then as argued in $(i)$ we have $Lr\leq e$. It follows that $L_\infty r\leq e$,
		and thus $L_\infty\rho\leq e$ holds by Lemma~\ref{L:23 leq}. By Lemma~\ref{L:muinfty}, we have that
		\[
		nL_\infty\rho\leq e
		\]
		holds for all $n\in\mathbb{N}$. This implies that $L_\infty\rho=0$, which in turn yields $L_\infty\wedge\rho=0$. Thus if $x\in B_{L_\infty}$, then by Lemma~\ref{L:largestelmt} we have
		\[
		|x|\wedge\rho\leq L_\infty\wedge\rho=0. 
		\]
		Hence we get $x\in B^d_{\rho}$.
		
		Next, note that by Theorem~\ref{T:3parts} and $(iii)$ we have $B_{\rho}^d\perp B_L^d$. It follows from Theorem~\ref{T:3parts} that $B_{\rho}^d\subseteq B_{L_{\mathcal{F}}}\oplus B_{L_\infty}$. We will show that $B_{\rho}^d\subseteq B_{L_\infty}$ by illustrating that $B_{\rho}^d\cap B_{L_{\mathcal{F}}}=\{0\}$. To this end, suppose that $B:=B_{\rho}^d\cap B_{L_{\mathcal{F}}}$ is a nontrivial band, and let $\mathbb{P}$ be its associated order projection. Fix $0<\epsilon<1$. Since $\epsilon L_{\mathcal{F}}^\ast$ is a positive invertible element in $B_{L_{\mathcal{F}}}$, we have that $x:=\mathbb{P}(\epsilon L_{\mathcal{F}}^\ast)$ is a positive invertible element in $B$. From $x\in B_{L_{\mathcal{F}}}$, we have $x\wedge L_\infty=0$ by Theorem~\ref{T:uF and u infty}$(ii)$. Thus from \eqref{E: limsup mult} we have
		\begin{align*}
			\underset{n\to\infty}{\limsup}(|a_n|^{1/n}x)=Lx=L_{\mathcal{F}}x=L_{\mathcal{F}}\mathbb{P}(\epsilon L_{\mathcal{F}}^\ast)=\epsilon\mathbb{P}(e)\ll e.
		\end{align*}
		Hence by \cite[Theorem~3.6(i)]{RoeSch} we get that $x\in\Omega$, and so $x\in B_{\rho}$. We thus have that $x=0$. This is a contradiction since we above deduced that $x$ was a positive invertible element in the nontrivial band $B$. Hence $B=\{0\}$, and so $B_{\rho}^d\subseteq B_{L_\infty}$. This proves $(ii)$.
		
		We conclude this proof by verifying $(iii)$. To this end, note that by Lemma~\ref{L:P_Fu=u_F} and the left-order continuity of $\bar{\mathbb{P}}_{L_{\mathcal{F}}}$ we have
		\[
		L_{\mathcal{F}}=\bar{\mathbb{P}}_{L_{\mathcal{F}}}(L)=\limsup_{n\to\infty}\,\mathbb{P}_{L_{\mathcal{F}}}(|a_n|)^{1/n}.
		\]
		Setting $b_n:=\mathbb{P}_{L_{\mathcal{F}}}(a_n)$ for every $n\in\mathbb{N}$, we have that $(b_n)_{n\ge 1}$ is order bounded and
		\[
		L_{\mathcal{F}}=\limsup_{n\to\infty}|b_n|^{1/n}.
		\]
		Then by \cite[Theorem~3.11(i)]{RoeSch}, we obtain
		\[
		(L_{\mathcal{F}})^*= \sup_{r\in\Omega}\, \mathbb{P}_{L_{\mathcal{F}}}(r)=\bar{\mathbb{P}}_{L_{\mathcal{F}}}(\rho).
		\]
		Using Theorem~\ref{T:3parts}, we have the following two band decompositions for $F$
		\[
		F=B_{L_{\mathcal{F}}}\oplus B_{L_\infty}\oplus B_L^d=B_{\rho_{\mathcal{F}}}\oplus B_{\rho_\infty}\oplus B^d_{\rho},
		\]
		and by $(i)$ and $(ii)$ we have $B_{L_{\mathcal{F}}}=B_{\rho_{\mathcal{F}}}$. Thus by Lemma~\ref{L:P_Fu=u_F} we obtain
		\[
		(L_{\mathcal{F}})^*=\underset{r\in\Omega}{\sup}\, \mathbb{P}_{L_{\mathcal{F}}}(r)=\bar{\mathbb{P}}_{L_{\mathcal{F}}}(\rho)=\bar{\mathbb{P}}_{\rho_{\mathcal{F}}}(\rho)=\rho_{\mathcal{F}},
		\]
		proving $(iii)$.
		
		The generalised inverse $L^{-1}$ of $L$ in $F^s$ is given by $(L_{\mathcal{F}})^* + \infty_L^d$, where $\infty_L^d$ is the largest element in $(B_L^d)^s$. By part $(iii)$ we have that $(L_{\mathcal{F}})^* = \rho_{\mathcal{F}}$, and by part $(i)$,  we have that $\infty_L^d$ is the largest element in $B_{\rho_\infty}^s$, which is $\rho_\infty$ by Lemma~\ref{L:largestelmt}. We can now conclude that $L^{-1}= \rho_{\mathcal{F}} +\rho_\infty = \rho$.
	\end{proof}
	
	\section{Differentiation of power series}\label{S:power series}
	
	We apply our theory of order differentiable functions to power series in this section.
	
	\begin{notation}
		Throughout the rest of the paper, $E$ denotes a universally complete complex vector lattice.
	\end{notation}
	
	\begin{theorem}\label{T:derivative of power series}
		Let $\sum_{n=0}^\infty a_n(z-c)^n$ be a power series on $E$ with spectrum of convergence $\Omega_1$. Then for the power series $\sum_{n=1}^\infty n a_n(z-c)^{n-1}$ with spectrum of convergence $\Omega_2$ the corresponding radii of convergence $\rho_1$ and $\rho_2$ are equal. In particular, $\Omega_1$ is order bounded in $F$ if and only if $\Omega_2$ is order bounded in $F$.
	\end{theorem}
	
	\begin{proof}
		Let $0<\epsilon<1$, and let $r\in \Omega_1$. Then the series $\sum_{n=0}^\infty a_n(\epsilon r)^n$ converges absolutely in order by \cite[Theorem~3.6(i)]{RoeSch}, since
		\[
		\limsup_{n\to\infty}|a_n|^{\frac{1}{n}}(\epsilon r)\le \epsilon e\ll e.
		\]
		Furthermore, we have $(\epsilon^2 r) (\epsilon r)^*\ll e$ and
		\[
		\limsup_{n\to\infty}n^{\frac{1}{n}}(\epsilon^2 r)(\epsilon r)^*\ll e,
		\]
		so it follows that the series
		\[
		\sum_{n=0}^\infty n((\epsilon^2 r)(\epsilon r)^*)^n
		\]
		also converges absolutely in order by \cite[Theorem~3.6(i)]{RoeSch}. Hence there is an $R\in E_+$ such that 
		\[
		n((\epsilon^2 r)(\epsilon r)^*)^n\le R
		\]
		for all $n\ge 0$. Thus 
		\begin{align*}
			n|a_n|(\epsilon^2 r)^{n-1} &= n|a_n|(\epsilon^2 r)^n(\epsilon^2 r)^* = \epsilon^{-1}n|a_n|(\epsilon^2 r)^n(\epsilon r)^* = \epsilon^{-1}n|a_n|(\epsilon^2 r)^n((\epsilon r)^*)^n(\epsilon r)^{n-1}\\&= \epsilon^{-1} n |a_n|(\epsilon^2 r)^n((\epsilon r)^*)^n(\epsilon r)^n(\epsilon r)^* \le \epsilon^{-2}r^*R|a_n|(\epsilon r)^n
		\end{align*}
		and so $\sum_{n=1}^\infty na_n(\epsilon^2 r)^{n-1}$ converges absolutely in order by comparison to $\sum_{n=0}^\infty a_n(\epsilon r)^n$, which yields $\epsilon^2r\in\Omega_2$.
		
		On the other hand, if $r\in\Omega_2$, then similarly we find that $\sum_{n=1}^\infty na_n(\epsilon r)^{n-1}$ converges absolutely in order. Since
		\[
		|a_n|(\epsilon r)^n\le n|a_n|(\epsilon r)^{n-1}(\epsilon r)
		\]
		for all $n\ge 1$, it follows that $\sum_{n=0}^\infty a_n(\epsilon r)^n$ converges absolutely in order, hence $\epsilon r\in\Omega_1$. 
		
		Next, for any $r\in\Omega_2$ we have shown that $\epsilon r\le\rho_1$, and by letting $\epsilon\to 1$, we conclude that $r\le\rho_1$. Hence $\rho_2\le \rho_1$. Moreover, if $r\in\Omega_1$, we have that $\epsilon^2 r\le\rho_2$ which yields $r\le\rho_2$ and hence $\rho_1\le\rho_2$. Therefore, $\rho_1=\rho_2$.
	\end{proof}
	
	The idea in the following lemma is used to naturally extend the definition of $\overset{\circ}{\Delta}(c,r)$ where $r\in F^s$.
	
	\begin{lemma}\label{L:largest order open set}
		Let $G$ be a uniformly complete Archimedean complex $\Phi$-algebra, $r \in G_+$ be an invertible element, and $c\in G$. Then $\overset{\circ}{\Delta}(c,r)$ is the largest order open set in $\bar{\Delta}(c,r)$.
	\end{lemma}
	
	\begin{proof}
		Let $z \in \overset{\circ}{\Delta}(c,r)$. Then $s:=\frac{1}{2}(r-|z-c|)$ is a positive invertible element and if $y \in \overset{\circ}{\Delta}(z,s)$, it follows that $r-|y-c|\ge r-|y-z|-|z-c|\ge r-s-|z-c|=\frac{1}{2}(r-|z-c|)$, which is positive and invertible by \cite[Theorem~11.1, Theorem~11.4]{dP}, hence $\overset{\circ}{\Delta}(z,s)\subseteq\overset{\circ}{\Delta}(c,r)$ and $\overset{\circ}{\Delta}(c,r)$ is an order open set. Suppose $z\in \bar{\Delta}(c,r)\setminus\overset{\circ}{\Delta}(c,r)$. Furthermore, suppose that there exist an invertible element $s \in G_+$ such that $\overset{\circ}{\Delta}(z,s) \subseteq \bar{\Delta}(c,r)$. It follows that $z+\frac{1}{2}\lambda s \in \overset{\circ}{\Delta}(z,s)$ for all $\lambda\in\mathbb{C}$ with $|\lambda|=1$, so that 
		\begin{align}\label{E:ineq1}
			|z-c+ {\textstyle\frac{1}{2}}\lambda s|\le r
		\end{align}
		for all $\lambda \in\mathbb{C}$ with $|\lambda|=1$. By the Kakutani representation theorem and then complexifying, there is a compact Hausdorff space $K$ such that the ideal generated by $r$ is lattice isomorphic to $C_\mathbb{C}(K)$, where $r$ corresponds to the constant function $\mathbf{1}$. If we identify $f,g \in C_\mathbb{C}(K)$ with $z-c$ and $s$, then there is $x \in K$ and $\lambda_0\in\mathbb{C}$ with $|\lambda_0|=1$ such that $\lambda_0f(x)= 1$. This now yields 
		\[
		|f+{\textstyle\frac{1}{2}}\overline{\lambda_0}g|(x)=|\overline{\lambda_0}(\lambda_0f+{\textstyle\frac{1}{2}}g)|(x)=|\lambda_0f+{\textstyle\frac{1}{2}}g|(x)=\lambda_0f(x)+{\textstyle\frac{1}{2}}g(x)>1, 
		\]
		which contradicts \eqref{E:ineq1}. Hence, $z$ is not an order interior point of $\bar{\Delta}(c,r)$ and $\overset{\circ}{\Delta}(c,r)$ is the largest order open set contained in $\bar{\Delta}(c,r)$.
	\end{proof}
	
	\begin{notation} 
		For $z,w\in F$ we use the notation $z\ll_w w$ when $w-z$ is a positive invertible element in $B_w$. Furthermore, for $y \in F^s$ and $x \in F$, by $x \ll y$ we mean that both $x \le y$ and $\mathbb{P}_{y_{\mathcal{F}}}(x) \ll_{y_{\mathcal{F}}} y_{\mathcal{F}}$. 
		In addition, if $y$ dominates a positive invertible element in $F$, the corresponding order disk $\overset{\circ}{\Delta}(c,y)$ will denote the largest order open set in $E$ that is contained in 
		\[
		\bar{\Delta}(c,y):=\{x \in E \colon |x-c|\le y\}=\{x \in E \colon \mathbb{P}_{y_{\mathcal{F}}}(|x-c|)\le y_{\mathcal{F}}\} + B_{y_\infty}.
		\]
		So, similarly to Lemma~\ref{L:largest order open set}, this set is defined by
		\[
		\overset{\circ}{\Delta}(c,y):=\{x\in E \colon \mathbb{P}_{y_{\mathcal{F}}}(|x-c|) \ll_{y_{\mathcal{F}}} y_{\mathcal{F}}\} + B_{y_\infty},
		\]
		and it is readily verified that this set is order open as well.
	\end{notation}

	If we want to think of a power series as a function, then we need to specify the domain, which necessarily must consist of all the $z \in E$ for which the power series converges in order.
	
	\begin{proposition}\label{P: domain power series}
		Let $\sum_{n=0}^\infty a_n(z-c)^n$ be a power series on $E$. Then the domain of the function $f(z):= \sum_{n=0}^\infty a_n(z-c)^n$ satisfies $\overset{\circ}{\Delta}(c,\rho) \subseteq \mathrm{dom}(f) \subseteq \bar{\Delta}(c,\rho)$. Furthermore, if $\rho$ dominates a positive invertible element, then $\overset{\circ}{\Delta}(c,\rho)$ is the largest order open set contained in $\mathrm{dom}(f)$.
	\end{proposition}
	
	\begin{proof}
		Let $z \in E$ be so that $|z-c| \ll \rho$. Then $\mathbb{P}_{\rho_{\mathcal{F}}}(|z-c|) \ll_{\rho_{\mathcal{F}}} \rho_{\mathcal{F}}$ and we can write $|z-c|=\mathbb{P}_{\rho_{\mathcal{F}}}(|z-c|)+\mathbb{P}_{\rho_{\infty}}(|z-c|)$. Let $L:=\limsup_{n\to\infty} |a_n|^{\frac{1}{n}}$. Note that $L$ is well-defined in $F^s$. It follows from Lemma~\ref{L:23} and \cite[Lemma~24$(iii)$]{AzouNasri}, which states that
		\[
		x(y\vee z)=(xy)\vee(xz)\qquad (0\leq x,y,z\in F^s),
		\]
		that $L|z-c|=\limsup_{n\to \infty}(|a_n|^{\frac{1}{n}}|z-c|)$. From Theorem~\ref{T:CHF} we obtain 
		\[
		\limsup_{n\to \infty}|a_n|^{\frac{1}{n}}|z-c|=L|z-c| = L\mathbb{P}_{\rho_{\mathcal{F}}}(|z-c|) = L_{\mathcal{F}}\mathbb{P}_{\rho_{\mathcal{F}}}(|z-c|) \ll e,
		\]
		so that $f(z)$ converges absolutely in order by \cite[Theorem~3.6(i)]{RoeSch}. Hence $z \in \mathrm{dom}(f)$.  
		If $z \in E$ is such that $\sum_{n=0}^\infty a_n(z-c)^n$ converges in order, then ${\textstyle\limsup_{n\to\infty}|a_n|^{\frac{1}{n}}|z-c| \le e}$ by \cite[Theorem~3.6(ii)]{RoeSch}. It follows that 
		that $L|z-c|=\limsup_{n\to \infty}(|a_n|^{\frac{1}{n}}|z-c|) \le e$. 
		
		Since $F$ is Archimedean and $L=L_{\mathcal{F}}+L_\infty$, we find that $L_\infty|z-c|=0$ by Lemma~\ref{L:muinfty}, so that $\mathbb{P}_{L_\infty}(|z-c|)=0$, and $L_{\mathcal{F}}|z-c|\le e$. By Theorem~\ref{T:CHF} and by Theorem~\ref{T:3parts}, it follows that $\mathbb{P}_{\rho_{\mathcal{F}}}(|z-c|)\le \rho_{\mathcal{F}}$, and 
		\begin{align*}
			|z-c|&=\mathbb{P}_{\rho_{\mathcal{F}}}(|z-c|)+\mathbb{P}_{\rho_{\infty}}(|z-c|)+\mathbb{P}_{\rho}^d(|z-c|)=\mathbb{P}_{\rho_{\mathcal{F}}}(|z-c|)+\mathbb{P}_{\rho_{\infty}}(|z-c|) \\&\le \rho_{\mathcal{F}} + \rho_\infty = \rho,
		\end{align*}
		as $\rho_\infty$ is the largest element in $(B_{\rho_\infty})^s$ by Lemma~\ref{L:largestelmt}. Hence $z \in \bar{\Delta}(c,\rho)$. 
		
		Suppose that $\rho$ dominates a positive invertible element. Since $\overset{\circ}{\Delta}(c,\rho)$ is the largest order open set contained in $\bar{\Delta}(c,\rho)$, it must therefore also be the largest order open set in $\mathrm{dom}(f)$ by an argument analogous to the proof of Lemma~\ref{L:largest order open set}.
	\end{proof}
	
	The remainder of this paper is devoted to proving that analytic functions are also holomorphic in our present setting.
	
	\begin{definition}
		Let $f \colon \mathrm{dom}(f) \to E$ and let $U \subseteq \mathrm{dom}(f)$ be an order open set. Then $f$ is said to be \emph{analytic} on $U$ if for every $c \in U$ there is a positive invertible element $r$ such that $\bar{\Delta}(c,r) \subseteq U$ and there exists a power series $\sum_{n=0}^\infty a_n(z-c)^n$ on $E$ that converges uniformly in order to $f(z)$ on $\bar{\Delta}(c,r)$. 
	\end{definition}
	
	In the next theorem we show when power series on universally complete complex vector lattices are order differentiable and consequently infinitely order differentiable. Part of our proof is an adaptation of a classical argument to the present setting.
	
	\begin{theorem}\label{T: derivative of power series 2}
		Let $f(z):=\sum_{n=0}^\infty a_n(z-c)^n$ and $g(z):=\sum_{n=1}^\infty n a_n(z-c)^{n-1}$ be functions defined by the respective power series with equal radius of convergence $\rho$ by Theorem~\ref{T:derivative of power series}. Then $f$ is order differentiable at $z_0$ if and only if $\rho$ dominates a positive invertible element and $z_0 \in \overset{\circ}{\Delta}(c,\rho)$.  In this case we have $f'(z_0) = g(z_0)$. 
	\end{theorem}
	
	\begin{proof} 
		Suppose $f$ is order differentiable at $z_0$. Then there is a positive invertible element $s$ such that $\overset{\circ}{\Delta}(z_0,s) \subseteq \mathrm{dom}(f)$. But then it follows from 
		\[
		{\textstyle \frac{1}{2}} s=|{\textstyle\frac{1}{2}}s| \le |z_0-c+{\textstyle\frac{1}{2}} s|+|z_0-c| \le 2\rho
		\]
		that $\rho$ dominates a positive invertible element. By Proposition~\ref{P: domain power series} we have that $\overset{\circ}{\Delta}(z_0,s) \subseteq \overset{\circ}{\Delta}(c,\rho)$, so that $z_0 \in \overset{\circ}{\Delta}(c,\rho)$. 
		
		Suppose that $\rho$ dominates a positive invertible element and $z_0 \in \overset{\circ}{\Delta}(c,\rho)$. Let $r$ be a positive invertible element such that $r \ll \rho$ and $z_0 \in \overset{\circ}{\Delta}(c,r)$. Let $h_\alpha \to 0$ be such that $z_0+h_\alpha \in \overset{\circ}{\Delta}(c,r)$, then it follows from \cite[Theorem~3.6(i)]{RoeSch} and Theorem~\ref{T:CHF} that $f(r+c)$ converges absolutely in order. Next for $n\ge 1$ define
		\[
		f_n(z):=\sum_{m=0}^\infty a_m\sum_{j=0}^{n-1}(z-c)^{n-1-j}(z_0-c)^j
		\]
		on $\overset{\circ}{\Delta}(c,r)$. Note that $f_n(z)$ is well defined since 
		\[
		\sum_{j=0}^{n-1}|z-c|^{n-1-j}|z_0-c|^j\le nr^{n-1}, 
		\]
		and $g(r+c)$ converges absolutely in order since $\rho_g=\rho_{f}$ by Theorem~\ref{T:derivative of power series}. Furthermore, we have that $f_n(z_0)=g(z_0)$ and as 
		\begin{align}\label{Eq: sum formula}
			\sum_{j=0}^{n-1}(z-c)^{n-j}(z_0-c)^j-\sum_{j=0}^{n-1}(z-c)^{n-j-1}(z_0-c)^{j+1}=(z-c)^n-(z_0-c)^n, 
		\end{align}
		we have
		\begin{align*}
			(z-z_0)f_n(z)&=(z-c)f_n(z)-(z_0-c)f_n(z)\\&=\sum_{m=0}^\infty a_m\sum_{j=0}^{n-1}(z-c)^{n-j}(z_0-c)^j-\sum_{m=0}^{\infty}a_m\sum_{j=0}^{n-1}(z-c)^{n-j-1}(z_0-c)^{j+1}\\&=f(z)-f(z_0).
		\end{align*}
		Using the equation in \eqref{Eq: sum formula} again, we find that
		\begin{align*}
			f_n(z)-g(z_0)&=\sum_{m=0}^\infty a_m\sum_{j=0}^{n-1}(z-c)^{n-1-j}(z_0-c)^j-\sum_{m=0}^\infty a_m\sum_{j=0}^{n-1}(z_0-c)^{n-1-j}(z_0-c)^j\\&=\sum_{m=0}^\infty a_m\sum_{j=0}^{n-1}(z_0-c)^j\Bigl((z-c)^{n-1-j}-(z_0-c)^{n-1-j}\Bigr)\\&=\sum_{m=0}^\infty a_m\sum_{j=0}^{n-2}(z_0-c)^j\Bigl((z-c)^{n-1-j}-(z_0-c)^{n-1-j}\Bigr).
		\end{align*}
		Since 
		\[
		\Bigl((z-c)-(z_0-c)\Bigr)\sum_{k=0}^{n-2-j}(z-c)^{n-2-j-k}(z_0-c)^{k}=(z-c)^{n-1-j}-(z_0-c)^{n-1-j},
		\]
		we find that
		\[
		f_n(z)-g(z_0)=(z-z_0)\sum_{m=0}^{\infty}a_m\sum_{j=0}^{n-2}\sum_{k=0}^{\ n-2-j}(z-c)^{n-2-j-k}(z_0-c)^{k+j}
		\]
		and 
		\begin{align*}
			\sum_{j=0}^{n-2}\sum_{k=0}^{\ n-2-j}|z-c|^{n-2-j-k}|z_0-c|^{k+j}&\le\sum_{j=0}^{n-2}\sum_{k=0}^{n-2-j}r^{n-2}=\sum_{j=0}^{n-2}(n-1-j)r^{n-2}\\&={\textstyle \frac{1}{2}}n(n-1)r^{n-2}.
		\end{align*}
		Since $h_\alpha\to 0$, there is a net $q_\beta\downarrow 0$ such that for all $\beta$ there is an $\alpha_0$ such that $|h_\alpha|\le q_\beta$ whenever $\alpha\ge \alpha_0$. Then for $\beta$ and such an $\alpha_0$, it follows by combining all of the above established identities that
		\begin{align*}
			|f(z_0+h_\alpha)-f(z_0)-h_\alpha g(z_0)|&=|h_\alpha f_n(z_0+h_\alpha)-h_\alpha g(z_0)|=|h_\alpha||f_n(z_0+h_\alpha)-g(z_0)|\\&\le
			{\textstyle \frac{1}{2}}|h_\alpha|q_\beta\sum_{k=0}^\infty |a_k| k(k-1)r^{k-2}
		\end{align*}
		whenever $\alpha \ge \alpha_0$. Note that this series converges by applying Theorem~\ref{T:derivative of power series} twice. Thus $f$ is order differentiable at $z_0$ where the derivative is $g(z_0)$. 
	\end{proof}
	
	We conclude this paper with the following result, which is a consequence of Theorem~\ref{T: derivative of power series 2}.
	
	\begin{theorem}
		Let $f\colon\mathrm{dom}(f)\to E$, and suppose that $U \subseteq \mathrm{dom}(f)$ is order open. If $f$ is analytic on $U$, then $f$ is holomorphic on $U$.
	\end{theorem}
	
	\begin{proof}
		Let $c \in U$. Then there is a positive invertible element $r\in E$ and a power series $\sum_{k=1}^\infty a_k(z-c)^k$ on $E$ such that $\sum_{k=1}^\infty a_k(z-c)^k$ converges uniformly to $f(z)$ on $\bar{\Delta}(c,r) \subseteq U$. It follows that $r \le \rho$, and by Theorem~\ref{T: derivative of power series 2}, the function $f$ is order differentiable at $c$. Hence $f$ is holomorphic on $U$.
		
		%\begin{comment}
		%Assume that $f$ is analytic on $U$. We wish to show that $f$ possesses an order pseudo-derivative at every point in $U$. To this end, let $c\in U$ be arbitrary. Since $f$
		%is analytic on $U$, there exists an $r>0$ and a power series $\sum_{n=0}^\infty a_n(z-c)^n$ such that
		%\begin{itemize}
		%\item[(i)] $\bar{\Delta}(c,r)\subseteq U$,
		%\item[(ii)] $\sum_{n=0}^{\infty}a_n(z-c)^n$ converges uniformly in order on $\bar{\Delta}(c,r)$, and
		%\item[(iii)] $f(z)=\sum_{n=0}^{\infty}a_n(z-c)^n$ for every $z\in\bar{\Delta}(c,r)$.
		%\end{itemize}
		%Then $f(r+c)$ converges in order since $r+c\in\bar{\Delta}(c,r)$. This implies that $f(z)$ converges uniformly on $\bar{\Delta}(c,r)$. Hence $r\in\Omega$, and thus $r\leq\rho$. By replacing $r$ with $\epsilon r$ for $0<\epsilon<1$ if necessary, we can suppose that
		%\[
		%\mathbb{P}_{\rho_F}(r)\ll_{\rho_F}\rho_F.
		%\]
		%By Theorem~\ref{T: derivative of power series 2}, there exists an order pseudo-derivative of $f$ at $z_0$ for each $z_0\in\bar{\Delta}(c,r)$. Therefore, $f$ is holomorphic on $U$.
		%\end{comment}
	\end{proof}
	
	\bibliography{holomorphic}
	\bibliographystyle{amsplain}
	
\end{document}